\documentclass{mcom-l}

\usepackage{amssymb}

\usepackage{graphicx}

\usepackage{color}
\usepackage[cmtip,all]{xy}

\usepackage{graphicx}

\newtheorem{theorem}{Theorem}[section]
\newtheorem{lemma}[theorem]{Lemma}

\theoremstyle{definition}

\theoremstyle{remark}
\newtheorem{rmk}[theorem]{Remark}

\numberwithin{equation}{section}

\renewcommand{\theequation}{\arabic{section}.\arabic{equation}}
\renewcommand{\thetable}{\arabic{section}.\arabic{table}}
\renewcommand{\thefigure}{\arabic{section}.\arabic{figure}}

\newcommand{\bx}{{\bf x} }

\newcommand{\p}{\partial}
\newcommand{\eps}{\varepsilon}

\newcommand{\be}{\begin{equation}}
\newcommand{\ee}{\end{equation}}
\newcommand{\ba}{\begin{array}}
\newcommand{\ea}{\end{array}}
\newcommand{\bea}{\begin{eqnarray}}
\newcommand{\eea}{\end{eqnarray}}
\newcommand{\beas}{\begin{eqnarray*}}
\newcommand{\eeas}{\end{eqnarray*}}

\begin{document}

\title[Super-resolution of TSSP for Dirac equation]{Super-resolution of time-splitting methods for the Dirac equation in the nonrelativistic regime}


\author{Weizhu Bao}
\address{Department of Mathematics, National University of Singapore, Singapore
	119076\\
	URL: {\tt http://blog.nus.edu.sg/matbwz/)}}
\email{matbaowz@nus.edu.sg}
\thanks{We acknowledge support from the Ministry of Education of Singapore grant  R-146-000-247-114 (W. Bao and J. Yin) and the NSFC grant No. 11771036 and 91630204 (Y. Cai).}

\author{Yongyong Cai}
\address{School of Mathematical Sciences, Beijing Normal University, 100875, P.R. China\\
and Beijing Computational Science Research Center,
	Beijing 100193, P. R. China}
\email{yongyong.cai@bnu.edu.cn}
\thanks{This work was partially done when the first author was visiting the Courant Institute for
	Mathematical Sciences in 2018.  Part of this work was done when the authors
visited the Institute for Mathematical Sciences, National University of Singapore, in 2019.}

\author{Jia Yin}
\address{NUS Graduate School for Integrative Sciences and Engineering (NGS),
	National University of Singapore, Singapore 117456}
\email{matyinj@nus.edu.sg}


\date{}

\dedicatory{}

\begin{abstract}
	We establish error bounds of the Lie-Trotter splitting ($S_1$) and Strang splitting ($S_2$) for the Dirac equation in the nonrelativistic regime in the absence of external magnetic potentials, with a small parameter $0<\eps\leq 1$ inversely proportional to the speed of light. In this regime, the solution propagates waves with $O(\eps^2)$  wavelength in time. Surprisingly, we find out that the splitting methods exhibit super-resolution, i.e. the methods can capture the solutions accurately even if the time step size $\tau$ is independent of $\eps$, while the wavelength in time is at $O(\eps^2)$.  $S_1$ shows $1/2$ order convergence uniformly with respect to $\eps$, by establishing that there are two independent error bounds $\tau + \eps$ and $\tau + \tau/\eps$. Moreover, if  $\tau$ is non-resonant, i.e. $\tau$ is away from certain region determined by $\eps$, $S_1$ would yield an improved uniform
	first order $O(\tau)$ error bound. In addition, we show $S_2$ is uniformly convergent with 1/2 order rate for general time step size $\tau$ and uniformly convergent with $3/2$ order rate for non-resonant time step size. Finally, numerical examples are reported to validate our findings.
\end{abstract}

\maketitle

\section{Introduction}\setcounter{equation}{0}
The splitting technique  introduced by Trotter in 1959 \cite{Trotter}  has been widely applied in analysis and numerical simulation \cite{ABB, BJM, BJM2, Carles, CG}, especially in computational quantum physics. In the Hamiltonian system and general ordinary differential equations (ODEs), the splitting approach has been shown to preserve the structural/geometric properties \cite{Hairer06,Verlet} and is superior in many applications.  Developments of splitting type methods in solving partial differential equations (PDEs) include utilization in Schr\"{o}dinger/nonlinear Schr\"{o}dinger equations \cite{ABB, BJM, BJM2, Carles, CG,Lubich,Tha}, Dirac/nonlinear Dirac equations \cite{BCJT2, BCJY, BY, LLS}, Maxwell-Dirac system \cite{BL, HJM}, Zakharov system \cite{BS, BSW, Gauckler, JMZ, JZ}, Stokes equation \cite{CHP}, and Enrenfest dynamics \cite{FJS}, etc.

When dealing with oscillatory problems,  the splitting method usually performs much better than traditional numerical methods  \cite{Hairer06,BJM}.
For instance,  in order to obtain ``correct'' observables of the Schr\"{o}dinger equation in the semiclassical regime, the time-splitting spectral method requires much weaker constraints on time step size and mesh size than the  finite difference methods \cite{BJM}. Similar properties have been observed for the nonlinear Schr\"{o}dinger equation (NLSE)/Gross-Pitaevskii equation (GPE) in the semiclassical regime \cite{ABB} and the Enrenfest dynamics \cite{FJS}. However, in general, splitting methods still suffer from the mesh size/time step constraints related to the high frequencies  in the aforementioned problems \cite{Bad,DT,JL}, i.e. in order to resolve a wave one needs to use a few grid points per wavelength. In this paper, we report a surprising finding that the splitting methods are uniformly accurate (w.r.t. the rapid oscillations),  when applied to the Dirac equation in the nonrelativistic  regime without external magnetic field. This fact reveals that there is no mesh size/time step restriction for splitting methods in this situation, e.g. the splitting methods have {\bf super-resolution} independent of the wavelength, which is highly nontrivial. In the rest of the paper, we will discuss the oscillatory Dirac equation in the nonrelativistic regime,  with conventional time splitting numerical approach and its super-resolution properties.

Proposed by British physicist Paul Dirac in 1928 \cite{Dirac}, the Dirac equation has now been extensively applied in the  study of the structures and/or dynamical properties of graphene, graphite, and other two-dimensional (2D) materials \cite{AMP, FW, NGPNG, NGMJZ}, as well as the relativistic effects of molecules in super intense lasers, e.g., attosecond lasers \cite{BCLL, FLB}.  Mathematically, the $d$-dimensional ($d = 1, 2, 3$) Dirac equation with external electro-magnetic potentials   \cite{BCJT2, BY} for the complex spinor vector field $\Psi:=\Psi(t,\bx)=(\psi_1(t,\bx),\psi_2(t,\bx), \psi_3(t,\bx), \psi_4(t,\bx))^T\in\mathbb{C}^4$ can be written as
\be\label{eq:dirac4}
i\partial_t\Psi =  \left(- \dfrac{i}{\eps}\sum_{j = 1}^{d}\alpha_j\partial_j +
\dfrac{1}{\varepsilon^2}\beta\right)\Psi
+ \left(V(t, \mathbf{x})I_4 - \sum_{j = 1}^{d}A_j(t, \mathbf{x})\alpha_j\right)\Psi,
\ee
for $\mathbf{x}\in\mathbb{R}^d$, $t>0$, with initial value
\be \label{eq:initial}
\Psi(t=0,\bx)=\Psi_0(\bx),\qquad \bx\in{\mathbb R}^d,
\ee
where $i=\sqrt{-1}$,  $t$ is time, $\bx=(x_1,\ldots,x_d)^T\in {\mathbb R}^d$
is the spatial coordinate vector, $\partial_j=\frac{\partial}{\partial x_j}$ ($j=1,\ldots,d$), $V:=V(t,\bx)$ and $A_j := A_j(t, \bx)$ ($j = 1, \ldots, d$) are the given real-valued electric and magnetic potentials,
respectively, $\eps \in (0,1]$ is a dimensionless parameter
inversely proportional to the speed of light. There are two important regimes for the Dirac equation \eqref{eq:dirac4}: the relativistic case $\eps=O(1)$ (wave speed
is comparable to the speed of light) and the nonrelativistic case $\eps\ll1$ (wave speed
is much less than the speed of light).
$I_n$ is the $n\times n$ identity matrix for $n\in {\mathbb{N}^*}$, and the $4\times 4$
matrices $\alpha_1$, $\alpha_2$, $\alpha_3$ and $\beta$ are
\be
\begin{aligned}
\label{alpha}
&\alpha_1=\left(\begin{array}{cc}
	\mathbf{0} & \sigma_1  \\
	\sigma_1 & \mathbf{0}  \\
\end{array}
\right),\qquad
\alpha_2=\left(\begin{array}{cc}
	\mathbf{0} & \sigma_2 \\
	\sigma_2 & \mathbf{0} \\
\end{array}
\right), \\
&\alpha_3=\left(\begin{array}{cc}
	\mathbf{0} & \sigma_3 \\
	\sigma_3 & \mathbf{0} \\
\end{array}
\right),\qquad
\beta=\left(\begin{array}{cc}
	I_{2}& \mathbf{0} \\
	\mathbf{0} & -I_{2} \\
\end{array}
\right),
\end{aligned}
\ee
where $\sigma_1$, $\sigma_2$,  $\sigma_3$
are the Pauli matrices
\be\label{Paulim}
\sigma_{1}=\left(
\begin{array}{cc}
	0 & 1  \\
	1 & 0  \\
\end{array}
\right), \qquad
\sigma_{2}=\left(
\begin{array}{cc}
	0 & -i \\
	i & 0 \\
\end{array}
\right),\qquad
\sigma_{3}=\left(
\begin{array}{cc}
	1 & 0 \\
	0 & -1 \\
\end{array}
\right).
\ee

In the relativistic  regime $\eps=O(1)$, extensive analytical and numerical studies have been carried out for the Dirac equation
\eqref{eq:dirac4} in the literature. In the analytical aspect, for the existence
and multiplicity of bound states and/or standing wave solutions, we
refer to \cite{Das, DK, ES, GGT, Gross, Ring} and references therein. In the
numerical aspect, many accurate and efficient numerical methods have been proposed and analyzed \cite{AL, MQ}, such as the finite difference time domain (FDTD)
methods \cite{ALSFB, NSG}, time-splitting Fourier pseudospectral (TSFP) method \cite{BCJT2, HJM}, exponential wave integrator Fourier pseudospectral (EWI-FP) method \cite{BCJT2}, and the Gaussian
beam method \cite{WHJY}, etc.

In the nonrelativistic  regime, as $ \eps\to 0^+$ ,
the Dirac equation \eqref{eq:dirac4} converges to Pauli equation \cite{BMP, Hunziker} or Schr\"{o}dinger equation \cite{Bad, BMP}, and
the solution propagates waves with wavelength O($\eps^2$) in time and  O(1) in space, respectively.
The highly oscillatory nature of the solution in time brings severe difficulties
in numerical computation in the nonrelativistic  regime, i.e. when $0 < \eps \ll 1$. In fact, it would cause the time step size $\tau$ to be
strictly dependent on $\eps$ in order to capture the solution accurately.
Rigorous error estimates were established for the finite difference time domain method (FDTD),
exponential wave integrator Fourier pseudospectral method (EWI-FP) and time-splitting Fourier
pseudospectral method (TSFP) in this parameter regime \cite{BCJT2}. The error bounds suggested  $\tau=O(\eps^3)$ for FDTD and $\tau=O(\eps^2)$ for EWI-FP and TSFP.
A new fourth-order compact time-splitting method ($S_\text{4c}$) was recently put forward
to improve the efficiency and accuracy \cite{BY}. Moreover,
a uniformly accurate multiscale time integrator pseudospectral method was proposed
and analyzed for the Dirac equation in the nonrelativistic  regime, where the errors are uniform with respect to $\eps\in(0,1]$ \cite{BCJT}, allowing for $\eps$-independent time step $\tau$.

From the analysis in \cite{BCJT2}, the error bounds for second order Strang splitting TSFP (also called as $S_2$ later in this paper) depends on the small parameter $\eps$
as $\tau^2/\eps^4$. Surprisingly, through our extensive numerical experiments, we find out that if the magnetic potentials
$A_j \equiv 0$ for $j = 1, \ldots, d$ in \eqref{eq:dirac4},  the errors of TSFP are then independent of $\eps$ and uniform w.r.t. $\eps$, i.e., $S_2$ for Dirac equation \eqref{eq:dirac4} without magnetic potentials $A_j$ has super-resolution w.r.t. $\eps$.
In such case, \eqref{eq:dirac4}
reduces to ($d = 1, 2, 3$)
\be\label{eq:dirac4_nomag}
i\partial_t\Psi(t, \bx) =  \left(- \dfrac{i}{\eps}\sum_{j = 1}^{d}\alpha_j\partial_j +
\dfrac{1}{\varepsilon^2}\beta
+ V(t, \mathbf{x})I_4\right)\Psi(t, \bx), \quad \bx\in{\mathbb R}^d, \; t>0,
\ee
with the initial value  given in \eqref{eq:initial}. In lower dimensions ($d=1,2$), the four component Dirac equation \eqref{eq:dirac4_nomag} can be reduced to the following two-component form for
$\Phi(t,\bx)=(\phi_1(t,\bx),\phi_2(t,\bx))^T\in \Bbb C^2$ ($d = 1, 2$) \cite{BCJT2}:
\be\label{eq:dirac2_nomag}
i\partial_t\Phi(t,\bx)=\left(-\frac{i}{\eps}\sum_{j=1}^{d}\sigma_j
\partial_j+\frac{1}{\eps^2}\sigma_3+V(t,\bx)I_2\right)
\Phi(t,\bx), \quad \bx\in{\mathbb R}^d, \; t>0,
\ee
with initial value
\be\label{eq:initial11}
\Phi(t=0,\bx)=\Phi_0(\bx),\qquad \bx\in{\mathbb R}^d.
\ee
The two component form \eqref{eq:dirac2_nomag} is widely used in lower dimensions $d=1,2$ due to its simplicity compared to the four component form \eqref{eq:dirac4_nomag}.

Our extensive numerical studies and theoretical analysis show that for first-order,
second-order, and even higher order time-splitting Fourier pseudospectral methods, there are always
uniform error bounds w.r.t. $\eps\in(0, 1]$. In other words, the splitting methods can capture the solutions accurately even if the time step size $\tau$ is independent of $\eps$, i.e.
they exhibit $\eps$-independent {\bf super-resolution}. As the \textbf{super-resolution} here suggests independence of the oscillation wavelength, it is even stronger than the `super-resolution' in \cite{DT} for the Schr\"{o}dinger equation in the semiclassical regime, where the restriction on the time steps is still related to the wavelength, but not so strict as the resolution of the oscillation by fixed number of points per wavelength.
This super-resolution property of the splitting methods makes them more efficient and reliable for solving the Dirac equation without magnetic potentials in the nonrelativisitc regime,  compared to other numerical approaches in the literature. In the sequel, we will study rigorously the super-resolution phenomenon for first-order ($S_1$) and second-order ($S_2$) time-splitting methods, and  present numerical results to validate the conclusions.

The rest of the paper is organized as follows.
In section 2, we review the first and second order time-splitting methods for
the Dirac equation in the nonrelativistic  regime without magnetic potential, and state the main results. In section 3 and section 4 respectively, detailed proofs for the uniform error bounds and improved uniform error bounds
are  presented. Section 5 is devoted to numerical tests, and finally,  some concluding remarks are
drawn in section 6.
Throughout the paper, we adopt the
standard Sobolev spaces  and the corresponding norms. Meanwhile, $A \lesssim B$ is used with
the meaning that there exists a generic constant $C > 0$ independent of $\eps$ and $\tau$,
such that $|A|\le C\,B$. $A \lesssim_\delta B$ has a similar meaning that there exists a constant $C_\delta>0$ dependent on $\delta$ but independent of $\eps$ and $\tau$, such that $|A|\le C_\delta\,B$.

\section{Time-splitting methods and main results}
In this section, we recall the first and second order time-splitting methods applied to the Dirac equation and state the main results of this paper. For simplicity of presentation, we only carry out the splitting methods and corresponding analysis for \eqref{eq:dirac2_nomag} in 1D ($d = 1$). Generalization to \eqref{eq:dirac4_nomag} and/or higher dimensions is straightforward and results remain valid without modifications (see Appendix).

\subsection{Time-splitting methods}
Denote the Hermitian operator
\be\label{eq:op:def}
\mathcal{T}^\eps=-i\eps\sigma_1\partial_x+\sigma_3, \quad x\in \Bbb R,
\ee
and the Dirac equation \eqref{eq:dirac2_nomag} in 1D can be written as
\be\label{eq:dirac_nomag}
i\p_t\Phi(t,x)=\frac{1}{\eps^2}\mathcal{T}^\eps\Phi(t,x)+V(t,x)\Phi(t,x), \quad x\in\mathbb{R},
\ee
with initial value
\be\label{eq:dirac_initial}
\Phi(0,x)=\Phi_0(x), \quad x\in\mathbb{R}.
\ee

Choose $\tau > 0$ to be the time step size and $t_{n}=n\tau$ for $n = 0, 1, ...$ as the time steps.
Denote $\Phi^n(x)$ as the numerical approximation of $\Phi(t_n,x)$,
where $\Phi(t,x)$ is the exact solution to \eqref{eq:dirac_nomag} with \eqref{eq:dirac_initial}, then the first-order and
second-order time-splitting methods can be expressed as follows.

{\bf First-order splitting (Lie-Trotter splitting)}. The discrete-in-time first-order splitting ($S_1$) is written as \cite{Trotter}
\be\label{eq:Lie-trotter}
\Phi^{n+1}(x)=e^{-\frac{i\tau}{\eps^2}\mathcal{T}^\eps}e^{-i\int_{t_n}^{t_{n+1}}V(s,x)\,ds}\Phi^n(x), \quad x\in\mathbb{R},
\ee
with $\Phi^0(x) = \Phi_0(x)$.

{\bf Second-order splitting (Strang splitting)}. The discrete-in-time second-order splitting ($S_2$) is written as \cite{Strang}
\be\label{eq:Strang}
\Phi^{n+1}(x)=e^{-\frac{i\tau}{2\eps^2}\mathcal{T}^\eps}e^{-i\int_{t_n}^{t_{n+1}}V(s,x)\,ds}e^{-\frac{i\tau}{2\eps^2}\mathcal{T}^\eps}\Phi^n(x), \quad x\in\mathbb{R}.
\ee
with $\Phi^0(x) = \Phi_0(x)$.

Then the main results of this paper can be summarized below.
\subsection{Uniform error bounds}
For any $T > 0$,
we are going to consider smooth enough solutions, i.e. we assume the electric potential satisfies
\begin{equation*}
(A)\hskip 1.4cm V(t,x)\in W^{m,\infty}([0,T]; L^\infty(\Bbb R))\cap
L^\infty([0,T]; W^{2m+m_*, \infty}(\Bbb R)), \hskip 1.4cm \end{equation*}
with $m\in\mathbb{N}^*$, $m_*\in\{0,1\}$. In addition, we assume the exact solution $\Phi(t, x)$ satisfies
\begin{equation*}
(B)\hskip 1.3cm \Phi(t, x)\in L^\infty([0, T], (H^{2m+m_*}(\mathbb{R}))^2), \quad m\in\mathbb{N}^*,\quad m_*\in\{0,1\}. \hskip8cm
\end{equation*}
We remark here that if the initial value $\Phi_0(\bx)\in (H^{2m+m_*}(\mathbb{R}))^2$, then  condition $(B)$ is implied by condition $(A)$.

For the numerical approximation $\Phi^n(x)$ obtained from $S_1$ \eqref{eq:Lie-trotter} or $S_2$ \eqref{eq:Strang}, we introduce the error function
\be\label{eq:en}
{\bf e}^n(x) = \Phi(t_n, x) - \Phi^n(x), \quad 0\leq n\leq \frac{T}{\tau},
\ee
then the following error estimates hold.
\begin{theorem}\label{thm:lie}
	Let $\Phi^n(x)$ be the numerical approximation obtained from $S_1$ \eqref{eq:Lie-trotter},
	then under the assumptions $(A)$ and $(B)$ with $m = 1$ and $m_*=0$,
	we have the following error estimates
	\be
	\|{\bf e}^n(x)\|_{L^2}\lesssim \tau+\eps,\quad \|{\bf e}^n(x)\|_{L^2}\lesssim \tau+\tau/\eps, \quad 0\le n\le\frac{T}{\tau}.
	\ee
	As a result, there is a uniform error bound for $S_1$
	\be
	\|{\bf e}^n(x)\|_{L^2}\lesssim \tau + \max_{0<\eps\leq 1}\min\{\eps, \tau/\eps\} \lesssim \sqrt{\tau}, \quad 0\le n\le\frac{T}{\tau}.
	\ee
\end{theorem}

\begin{theorem}\label{thm:strang}
	Let $\Phi^n(x)$ be the
	numerical approximation obtained from $S_2$ \eqref{eq:Strang}, then
	under the assumptions $(A)$ and $(B)$ with $m = 2$ and $m_*=0$,
	we have the following error estimates
	\be
	\|{\bf e}^n(x)\|_{L^2}\lesssim \tau^2+\eps,\quad \|{\bf e}^n(x)\|_{L^2}\lesssim \tau^2+\tau^2/\eps^3, \quad 0\le n\le\frac{T}{\tau}.
	\ee
	As a result, there is a uniform error bound for $S_2$
	\be
	\|{\bf e}^n(x)\|_{L^2}\lesssim \tau^2 + \max_{0<\eps\leq 1}\min\{\eps, \tau^2/\eps^3\}\lesssim \sqrt{\tau},  \quad 0\leq n\leq\frac{T}{\tau}.
	\ee
\end{theorem}
\begin{rmk}\label{remark:err}
	The error bounds in Theorem \ref{thm:lie} can be expressed as
	\be
	\|{\bf e}^n(x)\|_{L^2}\leq (C_1 + C_2T)\max_{t\in[0, T]}\|\Phi(t, x)\|_{H^2}\left(\tau + \max_{0<\eps\leq 1}\min\{\eps, \tau/\eps\}\right),
	\ee
	and the  error estimates in Theorem \ref{thm:strang} can be restated as
	\be
	\|{\bf e}^n(x)\|_{L^2}\leq (C_3 + C_4T)\max_{t\in[0, T]}\|\Phi(t, x)\|_{H^4}\left(\tau^2 + \max_{0<\eps\leq 1}\min\{\eps, \tau^2/\eps^3\}\right),
	\ee
	for $0\leq n\leq\frac{T}{\tau}$, where $C_1$, $C_2$, $C_3$ and $C_4$ are constants depending only on $V(t,x)$. We notice here that the error constants are linear in $T$.
\end{rmk}

We note that higher order time-splitting methods also share the super-resolution property, but for simplicity, we only focus on $S_1$ and $S_2$ here. Remark \ref{remark:err} could be easily derived by examining the proofs of Theorems \ref{thm:lie} \& \ref{thm:strang}, and the details will be skipped.

\subsection{Improved uniform error bounds for non-resonant time steps}
In the Dirac equation \eqref{eq:dirac2_nomag} or \eqref{eq:dirac4_nomag}, the leading term is $\frac{1}{\eps^2}\sigma_3\Phi$ or $\frac{1}{\eps^2}\beta\Psi$, which suggests the solution exhibits almost periodicity in time  with periods $2k\pi\eps^2$ ($k\in\mathbb{N}^*$, the periods of $e^{-i\sigma_3/\eps^2}$ and $e^{-i\beta/\eps^2}$).
From numerical results, we observe the errors behave much better compared to the results in Theorems \ref{thm:lie}\& \ref{thm:strang}, when $2\tau$ is away from  the leading temporal oscillation periods $2k\pi\eps^2$. In fact, for given $0<\delta\leq1$, define
\be \label{tau_range}
\mathcal{A}_\delta(\eps) := \bigcup_{k = 0}^\infty\left[\eps^2k\pi + \eps^2\arcsin\delta, \eps^2(k+1)\pi - \eps^2\arcsin\delta\right], \quad 0 < \eps \leq 1,
\ee
and the errors of $S_1$ and $S_2$ can be improved compared to the previous subsection when $\tau\in \mathcal{A}_\delta(\eps)$.
To illustrate $\mathcal{A}_\delta(\eps)$, we show in Figure \ref{fig:axis} for $\eps = 1$ and $\eps = 0.5$ with fixed $\delta = 0.15$.
\begin{figure}
	\caption{Illustration of non-resonant time steps $\mathcal{A}_\delta(\eps)$ with $\delta = 0.15$ for
		(a) $\eps = 1$ and (b) $\eps = 0.5$.}
	\vspace{5pt}
	\includegraphics[width=1\textwidth]{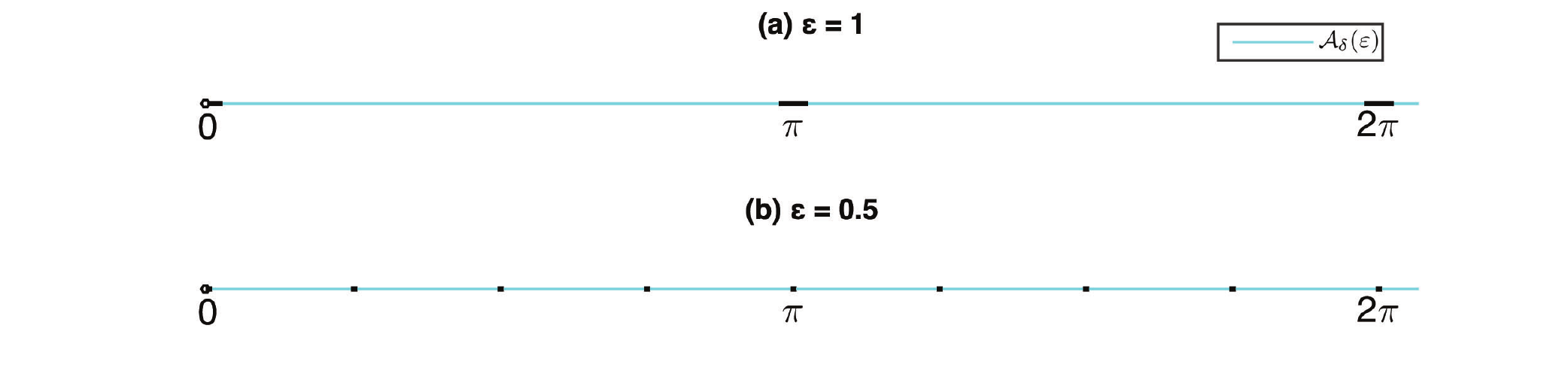}
	\label{fig:axis}
\end{figure}

For $\tau\in\mathcal{A}_\delta(\eps)$, we can derive improved uniform error bounds for the two splitting methods as shown in the following two theorems.
\begin{theorem}\label{thm:lie2}
	Let $\Phi^n(x)$ be the
	numerical approximation obtained from $S_1$ \eqref{eq:Lie-trotter}.
	If the time step size $\tau$ is non-resonant, i.e. there exists $0 < \delta \leq 1$, such that $\tau\in \mathcal{A}_\delta(\eps)$,  under the assumptions $(A)$  and $(B)$ with $m =1$ and $m_*=1$,
	we have an improved uniform error bound
	\be
	\|{\bf e}^n(x)\|_{L^2}\lesssim_\delta \tau, \quad 0\le n\le\frac{T}{\tau}.
	\ee
\end{theorem}

\begin{theorem}\label{thm:strang2}
	Let $\Phi^n(x)$ be the
	numerical approximation obtained from $S_2$ \eqref{eq:Strang}.
	If the time step size $\tau$ is non-resonant, i.e. there exists $0 < \delta \leq 1$, such that $\tau\in \mathcal{A}_\delta(\eps)$,  under the assumptions $(A)$ and $(B)$ with $m = 2$ and  $m_*=1$,
	we assume an extra regularity	$V(t,x)\in W^{1,\infty}([0,T];H^3(\Bbb R))$ and then the following two error estimates hold
	\be
	\|{\bf e}^n(x)\|_{L^2}\lesssim_\delta \tau^2+\tau\eps,\quad \|{\bf e}^n(x)\|_{L^2}\lesssim_\delta \tau^2+\tau^2/\eps, \quad 0\le n\le\frac{T}{\tau}.
	\ee
	As a result, there is an improved uniform error bound for $S_2$
	\be
	\|{\bf e}^n(x)\|_{L^2}\lesssim_\delta \tau^2 + \max_{0<\eps\leq 1}\min\{\tau\eps, \tau^2/\eps\}\lesssim_\delta \tau^{3/2}, \quad 0\leq n\leq\frac{T}{\tau}.
	\ee
\end{theorem}
\begin{rmk} In Theorems \ref{thm:lie2} and \ref{thm:strang2}, the constants in the error estimates depend on $\delta$ and the proof in the paper suggests that the constants are bounded from above by $\frac{T}{\tau}C$ and $\frac{2}{\delta}C$ with some common factor $C$ independent of $\delta$ and $\tau$. The optimality of the uniform error bounds in Theorems \ref{thm:lie2} and \ref{thm:strang2} will be verified by numerical examples presented in  section 5.
\end{rmk}

\begin{rmk} The results in Theorems \ref{thm:lie}, \ref{thm:strang}, \ref{thm:lie2}\&\ref{thm:strang2} can be generalized to higher dimensions ($d=2,3$) for Dirac equation \eqref {eq:dirac4_nomag}/\eqref{eq:dirac2_nomag} by the similar arguments. We will sketch a proof in Appendix.
\end{rmk}

\section{Proof of Theorems \ref{thm:lie} and \ref{thm:strang}}\label{sec:general}
In this section, we prove the uniform error bounds for the splitting methods $S_1$ and $S_2$.
As $\mathcal{T}^\eps$ is diagonalizable in the phase space (Fourier domain), it can be decomposed as \cite{BCJT,BCJT2,BMP}
\be\label{eq:dec}
\mathcal{T}^\eps=\sqrt{Id-\eps^2\Delta}\;\Pi_+^\eps-\sqrt{Id-\eps^2\Delta}\;\Pi_-^\eps,
\ee
where $\Delta=\p_{xx}$ is the Laplace operator in 1D and $Id$ is the  identity operator.
$\Pi_+^\eps$ and $\Pi_-^\eps$ are projectors defined as
\be\label{eq:pipm}
\Pi_+^\eps=\frac12\left[I_2+\left(Id-\eps^2\Delta\right)^{-1/2}\mathcal{T}^\eps\right],\quad \Pi_-^\eps=\frac12\left[I_2-\left(Id-\eps^2\Delta\right)^{-1/2}\mathcal{T}^\eps\right].
\ee
It is straightforward to see that $\Pi_+^\eps+\Pi_-^\eps=I_2$, and
$\Pi_+^\eps\Pi_-^\eps=\Pi_-^\eps\Pi_+^\eps={\bf 0}$, $(\Pi_{\pm}^\eps)^2=\Pi_{\pm}^\eps$. Furthermore, through Taylor expansion, we have \cite{BMP}
\begin{align}\label{eq:pi+}
\Pi_+^\eps=\Pi_+^0+\eps\mathcal{R}_1=\Pi_+^0-i\frac{\eps}{2}\sigma_1\partial_x+\eps^2\mathcal{R}_2,\quad \Pi_+^0=\text{diag}(1,0),\\
\label{eq:pi-}
\Pi_-^\eps=\Pi_-^0-\eps\mathcal{R}_1=\Pi_-^0+i\frac{\eps}{2}\sigma_1\partial_x-\eps^2\mathcal{R}_2,\quad \Pi_-^0=\text{diag}(0,1),
\end{align}
where $\mathcal{R}_1:(H^m(\Bbb R))^2\to (H^{m-1}(\Bbb R))^2$ for $m\geq 1$, $m \in \mathbb{N}^*$, and $\mathcal{R}_2:(H^m(\Bbb R))^2\to (H^{m-2}(\Bbb R))^2$ for $m\geq 2$, $m \in \mathbb{N}^*$ are uniformly bounded operators with respect to $\eps$.

To help capture the features of solutions, denote
\be\label{eq:ddef}
\mathcal{D}^\eps=\frac{1}{\eps^2}(\sqrt{Id-\eps^2\Delta}-Id)
=-(\sqrt{Id-\eps^2\Delta}+Id)^{-1}\Delta,\ee where $\mathcal{D}^\eps$ is a uniformly bounded operator
with respect to $\eps$ from $(H^{m}(\Bbb R))^2$ to $(H^{m-2}(\Bbb R))^2$ for $m\geq 2$,
then we have the decomposition for the unitary evolution  operator $e^{it\mathcal{T}^\eps/\eps^2}$ as \cite{BCJT,CaiW2}
\be\label{eq:dec:ev}
e^{\frac{it}{\eps^2}\mathcal{T}^\eps}  =
e^{\frac{it}{\eps^2}(\sqrt{Id-\eps^2\Delta}\;\Pi_+^\eps-\sqrt{Id-\eps^2\Delta}\;\Pi_-^\eps)}
= e^{it/\eps^2}e^{it\mathcal{D}^\eps}\;\Pi_+^\eps + e^{-it/\eps^2}e^{-it\mathcal{D}^\eps}\;\Pi_-^\eps.
\ee

For the ease of the proof, we first introduce the following two lemmas for the Lie-Trotter splitting $S_1$ \eqref{eq:Lie-trotter} and the Strang splitting $S_2$ \eqref{eq:Strang}, respectively. For simplicity, we denote $V(t) := V(t, x)$, and $\Phi(t) := \Phi(t, x)$ in short.\\

\begin{lemma}\label{lemma:lie}
	Let $\Phi^n(x)$ be the
	numerical approximation obtained from the Lie-Trotter splitting $S_1$ \eqref{eq:Lie-trotter},
	then under the assumptions $(A)$ and $(B)$ with $m = 1$ and $m_* = 0$, we have
	\be
	{\bf e}^{n+1}(x) = e^{-\frac{i\tau}{\eps^2}\mathcal{T}^\eps}e^{-i\int_{t_n}^{t_{n+1}}V(s, x)ds}{\bf e}^n(x) + \eta_1^n(x) + \eta_2^n(x), \quad 0\leq n\leq \frac{T}{\tau} - 1,
	\ee
	with $\|\eta_1^n(x)\|_{L^2}\lesssim\tau^2$, $\eta_2^n(x) = -ie^{-\frac{i\tau}{\eps^2}\mathcal{T}^\eps}\left(\int_0^\tau f_2^n(s)ds - \tau f_2^n(0)\right)$, where
	\begin{align}\label{eq:f2nlie}
	f_2^n(s) = &e^{i2s/\eps^2}e^{is\mathcal{D}^\eps}\Pi_+^\eps\left(V(t_n)\Pi_-^\eps e^{is\mathcal{D}^\eps}\Phi(t_n)\right) \nonumber\\
	&+ e^{-i2s/\eps^2}e^{-is\mathcal{D}^\eps}\Pi_-^\eps\left(V(t_n)\Pi_+^\eps e^{-is\mathcal{D}^\eps}\Phi(t_n)\right).
	\end{align}
\end{lemma}
\begin{proof}
	From the definition of ${\bf e}^n(x)$, noticing the Lie-Trotter splitting formula \eqref{eq:Lie-trotter}, we have 	\be\label{eq:rec1}
	{\bf e}^{n+1}(x)=e^{-\frac{i\tau}{\eps^2}\mathcal{T}^\eps}e^{-i\int_{t_n}^{t_{n+1}}V(s,x)\,ds}{\bf e}^n(x)+\eta^n(x), \quad 0\leq n \leq \frac{T}{\tau} - 1, \quad x\in{\Bbb R},
	\ee
	where $\eta^n(x)$ is the local truncation error defined as
	\be\label{eq:loc1}
	\eta^n(x)=\Phi(t_{n+1},x)-e^{-\frac{i\tau}{\eps^2}\mathcal{T}^\eps}e^{-i\int_{t_n}^{t_{n+1}}V(s,x)\,ds}\Phi(t_n,x), \quad x\in{\Bbb R}.
	\ee
	Noticing \eqref{eq:dirac_nomag}, applying Duhamel's principle, we derive
	\be
	\label{eq:Duhamel}
	\Phi(t_{n + 1}, x) = e^{-\frac{i\tau}{\eps^2}\mathcal{T}^\eps}\Phi(t_n, x) - i\int_0^\tau
	e^{-\frac{i(\tau - s)}{\eps^2}\mathcal{T}^\eps}V(t_n + s, x)\Phi(t_n + s, x)ds,
	\ee
	while Taylor expansion gives
	\begin{align}\label{eq:tl:s1}
	&e^{-\frac{i\tau}{\eps^2}\mathcal{T}^\eps}e^{-i\int_{t_n}^{t_{n+1}}V(s,x)\,ds}\Phi(t_n,x)\nonumber\\
	&=e^{-\frac{i\tau}{\eps^2}\mathcal{T}^\eps}\left(1-i\int_{t_n}^{t_{n+1}}V(s,x)\,ds+O(\tau^2)\right)\Phi(t_n,x).
	\end{align}
	Combining \eqref{eq:Duhamel}, \eqref{eq:tl:s1} and \eqref{eq:loc1}, we get
	\begin{align}
	\eta^n(x)
	=&\tau ie^{-\frac{i\tau}{\eps^2}\mathcal{T}^\eps}V(t_n)\Phi(t_n)-
	i\int_0^{\tau}e^{-\frac{i(\tau-s)}{\eps^2}\mathcal{T}^\eps}\left(V(t_n)
	e^{-\frac{is}{\eps^2}\mathcal{T}^\eps}\Phi(t_n)\right)\,ds\nonumber\\&+ \sum_{j=1}^2{R}_j^n(x),\label{eq:S1_eta}
	\end{align}
	where
	\begin{align*}
	{R}_1^n(x) &= e^{-\frac{i\tau}{\eps^2}\mathcal{T}^\eps}\left(\lambda_1^n(x)+\lambda_2^n(x)\right)\Phi(t_n,x),\\
	{R}_2^n(x)
	&=-i\int_0^{\tau}e^{-\frac{i(\tau-s)}{\eps^2}\mathcal{T}^\eps}\left(V(t_n )
	\lambda_4^{n}(s,x)+\lambda^n_3(s,x )
	\Phi(t_n+s,x)\right)\,ds,
	\end{align*}
	with
	\begin{align}
	\label{eq:Taylor}
	&\lambda_1^n(x)=e^{-i\int_{t_n}^{t_{n + 1}}V(s, x)ds} - \left(1 - i\int_{t_n}^{t_{n + 1}}V(s, x)ds\right),\\
	\label{eq:integ}
	&\lambda_2^n(x)=-i\int_{t_n}^{t_{n+1}}V(u,x)\,du+i\tau V(t_n,x), \\
	&\lambda_3^n(s,x)=V(t_{n}+s,x)-V(t_n,x), \quad 0\leq s\leq\tau,\label{eq:integ2}\\
	&\lambda_4^n(s,x)= -i\int_0^se^{-\frac{i(s-w)}{\eps^2}\mathcal{T}^\eps}\left(V(t_n+w,x)\Phi(t_n+w,x)\right)\,dw,\quad 0\leq s\leq\tau.\label{eq:Duhamel2}
	\end{align}
	It is easy to see that for $ 0\leq n\leq\frac{T}{\tau} - 1$,
	\begin{align*}
	&\|\lambda_1^n(x)\|_{L^\infty}\lesssim\tau^2\|V(t, x)\|_{L^\infty\left(L^\infty\right)}^2,\quad
	\|\lambda_2^n(x)\|_{L^\infty}\lesssim\tau^2\|\partial_tV(t, x)\|_{L^\infty\left(L^\infty\right)},\\	&\|\lambda_3^n(s,x)\|_{L^\infty([0,\tau];L^\infty)}\lesssim \tau\|\partial_tV(t, x)\|_{L^\infty\left(L^\infty\right)},\\
	&\|\lambda_4^n(s,x)\|_{L^\infty([0,\tau];(L^2)^2)}\lesssim \tau\|V(t, x)\|_{L^\infty\left(L^\infty\right)}\|\Phi(t, x)\|_{L^\infty\left((L^2)^2\right)},
	\end{align*}
	As a consequence, we obtain the following bounds for $0\leq n\leq\frac{T}{\tau}-1$,
	\begin{align}\label{eq:S1_R1}
	\|{R}_1^n(x)\|_{L^2}&\lesssim \left(\|\lambda_1^n(x)\|_{L^\infty}+\|\lambda_2^n(x)\|_{L^\infty}\right)\|\Phi(t_n)\|_{L^2}\lesssim \tau^2, \\
	\label{eq:S1_R3}
	\|{R}_2^n(x)\|_{L^2}&\lesssim \tau\bigg(\|V(t_n)\|_{L^\infty}\|\lambda_4^n(s,x)\|_{L^\infty([0,\tau];(L^2)^2)} \nonumber\\
	&\qquad + \|\lambda_3^n(s,x)\|_{L^\infty([0,\tau];L^\infty)}\|\Phi\|_{L^\infty\left((L^2)^2\right)}\bigg)\lesssim \tau^2.
	\end{align}
	Recalling $\eta_2^n(x)$ given in Lemma \ref{lemma:lie}, we introduce for $0\leq s\leq\tau$
	\be\label{eq:S1_f}
	f^n(s):=f^n(s,x)=e^{\frac{is}{\eps^2}\mathcal{T}^\eps}\left(V(t_n,x)e^{-\frac{is}{\eps^2}\mathcal{T}^\eps}\Phi(t_n,x)\right)=f_1^n(s)+f_2^n(s),
	\ee
	with $f_2^n$ given in \eqref{eq:f2nlie} and $f_1^n$ from the decomposition \eqref{eq:dec:ev} as
	\begin{align*}
	f_1^n(s) &= e^{is\mathcal{D}^\eps}\;\Pi_+^\eps\left(V(t_n)e^{-is\mathcal{D}^\eps}\Pi_+^\eps\Phi(t_n)\right)+
	e^{-is\mathcal{D}^\eps}\;\Pi_-^\eps\left(V(t_n)e^{is\mathcal{D}^\eps}\Pi_-^\eps\Phi(t_n)\right),	\end{align*}
	and then $\eta^n(x)$ \eqref{eq:S1_eta} can be written as
	\begin{align}\label{eq:split1}
	\eta^n(x) = &-ie^{-\frac{i\tau}{\eps^2}\mathcal{T}^\eps}\left(\int_0^\tau (f_1^n(s)+f_2^n(s))ds-\tau (f_1^n(0)+f_2^n(0))\right)\nonumber\\
	&+ R_1^n(x) + R_2^n(x).
	\end{align}
	Now, it is easy to verify that $\eta^n(x)=\eta_1^n(x)+\eta_2^n(x)$ with $\eta_2^n(x)$  given in Lemma \ref{lemma:lie} if we let
	\be
	\eta_1^n(x)=-ie^{-\frac{i\tau}{\eps^2}\mathcal{T}^\eps}\left(\int_0^\tau f_1^n(s)ds-\tau f_1^n(0)\right)+ R_1^n(x) + R_2^n(x) .
	\ee
	Noticing that
	\begin{align*}
	&\left\|e^{-\frac{i\tau}{\eps^2}\mathcal{T}^\eps}\left(\int_0^\tau f_1^n(s)ds-\tau f_1^n(0)\right)\right\|_{L^2}\\
	&\lesssim\; \tau^2\|\partial_sf_1^n(\cdot)\|_{L^\infty([0,\tau];(L^2)^2)}
	\lesssim\;
	\tau^2\|V(t_n)\|_{W^{2,\infty}}\|\Phi(t_n)\|_{H^2},
	\end{align*}
	recalling the regularity assumptions $(A)$ and $(B)$, combining  \eqref{eq:S1_R1} and \eqref{eq:S1_R3} , we can get
	\begin{align*}
	\|\eta_1^n(x)\|_{L^2} &\leq \|{R}_1^n(x)\|_{L^2} + \|{R}_2^n(x)\|_{L^2} +\left\|e^{-\frac{i\tau}{\eps^2}\mathcal{T}^\eps}\left(\int_0^\tau f_1^n(s)ds-\tau f_1^n(0)\right)\right\|_{L^2} \\
	&\lesssim \tau^2,
	\end{align*}
	which completes  the proof of Lemma \ref{lemma:lie}.
\end{proof}

\begin{lemma}\label{lemma:strang}
	Let $\Phi^n(x)$ be the
	numerical approximation obtained from the Strang splitting $S_2$ \eqref{eq:Strang},
	then under the assumptions $(A)$ and $(B)$ with $m = 2$ and $m_* = 0$, we have for $0\leq n\leq \frac{T}{\tau} - 1$,
	\be
	{\bf e}^{n+1}(x) = e^{-\frac{i\tau}{2\eps^2}\mathcal{T}^\eps}e^{-i\int_{t_n}^{t_{n+1}}V(s, x)ds}e^{-\frac{i\tau}{2\eps^2}\mathcal{T}^\eps}{\bf e}^n(x) + \eta_1^n(x) + \eta_2^n(x) + \eta_3^n(x),
	\ee
	with
	\begin{align}\label{eq:eta2:lemma}
	&\|\eta_1^n(x)\|_{L^2}\lesssim\tau^3,\quad \eta_2^n(x) = -ie^{-\frac{i\tau}{\eps^2}\mathcal{T}^\eps}\left(\int_0^\tau f_2^n(s)ds - \tau f_2^n(\tau/2)\right), \\
	&\eta_3^n(x) = -e^{-\frac{i\tau}{\eps^2}\mathcal{T}^\eps}\left(\int_0^\tau\int_0^s \sum_{j = 2}^4g_j^n(s, w)dwds - \frac{\tau^2}{2}\sum_{j = 2}^4g_j^n(\tau/2, \tau/2)\right),\label{eq:eta3:lemma}
	\end{align}
	where
	\begin{align}
	f_2^n(s) =& e^{\frac{i2s}{\eps^2}}e^{is\mathcal{D}^\eps}\;\Pi_+^\eps(V(t_n+s)e^{is\mathcal{D}^\eps}\Pi_-^\eps\Phi(t_n)\nonumber)\\
	&+
	e^{\frac{-i2s}{\eps^2}}e^{-is\mathcal{D}^\eps}\;\Pi_-^\eps(V(t_n+s)e^{-is\mathcal{D}^\eps}\Pi_+^\eps\Phi(t_n)),\label{eq:f2ns}\\
	g_2^n(s, w) =& e^{i2w/\eps^2}e^{is\mathcal{D}^\eps}\Pi_+^\eps\left(V(t_n)e^{-i(s - w)\mathcal{D}^\eps}\Pi_+^\eps\left(V(t_n)e^{iw\mathcal{D}^
		\eps}\Pi_-^\eps\Phi(t_n)\right)\right)\nonumber\\
	& + e^{-i2w/\eps^2}e^{-is\mathcal{D}^\eps}\Pi_-^\eps\left(V(t_n)e^{i(s - w)\mathcal{D}^\eps}\Pi_-^\eps\left(V(t_n)e^{-iw\mathcal{D}^
		\eps}\Pi_+^\eps\Phi(t_n)\right)\right),\label{eq:g2ns}\\
	g_3^n(s, w) =&e^{\frac{i2(s-w)}{\eps^2}} e^{is\mathcal{D}^\eps}\Pi_+^\eps\left(V(t_n)e^{i(s - w)\mathcal{D}^\eps}\Pi_-^\eps\left(V(t_n)e^{-iw\mathcal{D}^
		\eps}\Pi_+^\eps\Phi(t_n)\right)\right)\nonumber\\
	& + e^{-\frac{i2(s-w)}{\eps^2}}e^{-is\mathcal{D}^\eps}\Pi_-^\eps\left(V(t_n)e^{-i(s - w)\mathcal{D}^\eps}\Pi_+^\eps\left(V(t_n)e^{iw\mathcal{D}^
		\eps}\Pi_-^\eps\Phi(t_n)\right)\right),\label{eq:g3ns}\\
	g_4^n(s, w) =& e^{i2s/\eps^2} e^{is\mathcal{D}^\eps}\Pi_+^\eps\left(V(t_n)e^{i(s - w)\mathcal{D}^\eps}\Pi_-^\eps\left(V(t_n)e^{iw\mathcal{D}^
		\eps}\Pi_-^\eps\Phi(t_n)\right)\right)\nonumber\\
	& + e^{-i2s/\eps^2}e^{-is\mathcal{D}^\eps}\Pi_-^\eps\left(V(t_n)e^{-i(s - w)\mathcal{D}^\eps}\Pi_+^\eps\left(V(t_n)e^{-iw\mathcal{D}^
		\eps}\Pi_+^\eps\Phi(t_n)\right)\right).\label{eq:g4ns}
	\end{align}
	
\end{lemma}

\begin{proof}
	From the definition of ${\bf e}^n(x)$, noticing the Strang splitting formula \eqref{eq:Strang}, we have
	\be\label{eq:rec2}
	{\bf e}^{n+1}(x)=e^{-\frac{i\tau}{2\eps^2}\mathcal{T}^\eps}e^{-i\int_{t_n}^{t_{n+1}}V(s,x)\,ds}e^{-\frac{i\tau}{2\eps^2}\mathcal{T}^\eps}{\bf e}^n(x)+\eta^n(x), \quad x\in{\Bbb R},
	\ee
	where $\eta^n(x)$ is the local truncation error defined as
	\be\label{eq:loc2}
	\eta^n(x)=\Phi(t_{n+1},x)-e^{-\frac{i\tau}{2\eps^2}\mathcal{T}^\eps}e^{-i\int_{t_n}^{t_{n+1}}V(s,x)\,ds}e^{-\frac{i\tau}{2\eps^2}\mathcal{T}^\eps}\Phi(t_n,x), \quad x\in{\Bbb R}.
	\ee
	Similar to the $S_1$ case, repeatedly using Duhamel's principle and Taylor expansion, we can
	obtain
	\begin{align}\label{eq:taylor1}
	&\Phi(t_{n+1})\nonumber\\
	&=e^{-\frac{i\tau}{\eps^2}\mathcal{T}^\eps}\Phi(t_n)-i\int_0^\tau e^{-\frac{i(\tau-s)}{\eps^2}\mathcal{T}^\eps}\left(V(t_n+s)e^{-\frac{is}{\eps^2}\mathcal{T}^\eps}\Phi(t_n)\right)\,ds\nonumber\\
	&\quad-\int_0^\tau\int_0^se^{-\frac{i(\tau-s)}{\eps^2}\mathcal{T}^\eps}\left(V(t_n)e^{-\frac{i(s-w)}{\eps^2}\mathcal{T}^\eps}\left(V(t_n+w)\Phi(t_n+w)\right)\right)\,dw\,ds,\end{align}
	\begin{align}
	&e^{-\frac{i\tau}{2\eps^2}\mathcal{T}^\eps}e^{-i\int_{t_n}^{t_{n+1}}V(s)\,ds}e^{-\frac{i\tau}{2\eps^2}\mathcal{T}^\eps}\Phi(t_n)\nonumber\\ &=e^{-\frac{i\tau}{2\eps^2}\mathcal{T}^\eps}\left(1-i\int_{0}^{\tau}V(t_n+s)\,ds-\frac{1}{2}(\int_{0}^{\tau}V(t_n+s)\,ds)^2\right)
	e^{-\frac{i\tau}{2\eps^2}\mathcal{T}^\eps}\Phi(t_n)\nonumber\\
	&\quad+e^{-\frac{i\tau}{2\eps^2}\mathcal{T}^\eps}\left(O(\tau^3)\right)e^{-\frac{i\tau}{2\eps^2}\mathcal{T}^\eps}\Phi(t_n).\label{eq:taylor2}
	\end{align}
	Denoting
	\be
	\label{eq:fn:S2}f^n(s)=e^{\frac{is}{\eps^2}\mathcal{T}^\eps}\left(V(t_n+s,x)e^{-\frac{is}{\eps^2}\mathcal{T}^\eps}\Phi(t_n,x)\right),
	\ee
	for $0\leq s\leq\tau$, and
	\be
	g^n(s,w)=e^{\frac{is}{\eps^2}\mathcal{T}^\eps}\left(V(t_n,x)e^{-\frac{i(s-w)}{\eps^2}\mathcal{T}^\eps}\left(V(t_n,x)e^{-\frac{iw}{\eps^2}\mathcal{T}^\eps}\Phi(t_n,x)\right)\right),\label{eq:gn:S2}
	\ee
	for $0\leq s,w\leq \tau$, in view of \eqref{eq:taylor1} and \eqref{eq:taylor2},  $\eta^n(x)$ \eqref{eq:loc2} can be written as
	\begin{align}
	\eta^n(x)
	=&	
	-e^{-\frac{i\tau}{\eps^2}\mathcal{T}^\eps}\bigg[i\int_0^\tau f^n(s)\,ds-i\tau f^n\left(\frac{\tau}{2}\right)
	+\int_0^\tau\int_0^sg^n(s,w)\,dwds\nonumber\\
	&\qquad\qquad\qquad-\frac{\tau^2}{2}g^n\left(\frac{\tau}{2},\frac{\tau}{2}\right)\bigg]+\sum_{j=1}^2R_j^n(x),\label{eq:eta-S2:1}\end{align}
	where 	
	\begin{align}
	&R_1^n(x)=-e^{-\frac{i\tau}{2\eps^2}\mathcal{T}^\eps}(\lambda_1^n(x)+\lambda_2^n(x))e^{-\frac{i\tau}{2\eps^2}\mathcal{T}^\eps}\Phi(t_n,x),\nonumber\\
	&R_2^n(x)\nonumber\\
	&=-\int_0^\tau\int_0^se^{-\frac{i(\tau-s)}{\eps^2}\mathcal{T}^\eps}\bigg(V(t_n+s,x)e^{-\frac{i(s-w)}{\eps^2}\mathcal{T}^\eps}\left(V(t_n+w,x)\lambda_3^n(w,x)\right)\bigg)\,dw\,ds,\nonumber
	\end{align}
	with
	\begin{align*}
	&\lambda_1^n(x)=-i\left(\int_{0}^{\tau}V(t_n + s,x)\,ds-\tau V(t_n+\frac{\tau}{2},x)\right)
	-\frac{1}{2}
	\left(\int_0^{\tau}V(t_n+s,x)ds\right)^2\\
	&\qquad\qquad+\frac{1}{2}\tau^2V^2(t_n,x),\\
	&\lambda_2^n(x)=e^{-i\int_0^{\tau}V(t_n+s,x)ds}-1+i\int_0^{\tau}V(t_n+s,x)ds+\frac{1}{2}
	\left(\int_0^{\tau}V(t_n+s,x)ds\right)^2,\\
	&\lambda_3^n(w,x) = -i\int_0^we^{-\frac{i(w-u)}{\eps^2}\mathcal{T}^\eps}\left(V(t_n+u,x)\Phi(t_n+u,x)\right)\,du.
	\end{align*}
	It is easy to check that $\|\lambda_2^n(x)\|_{L^\infty}\lesssim \tau^3\|V(t, x)\|^3_{L^\infty\left(L^\infty\right)}$ and
	\begin{align*}
	&\|\lambda_1^n(x)\|_{L^\infty}\lesssim\tau^3\|\partial_{tt}V(t, x)\|_{L^\infty\left(L^\infty\right)}+\tau^3\|\partial_{t}V(t, x)\|_{L^\infty\left(L^\infty\right)}\|V(t, x)\|_{L^\infty\left(L^\infty\right)}, \\
	& \|\lambda_3^n(w,x)\|_{L^\infty([0,\tau];(L^2)^2)}\lesssim \tau \|V(t,x)\|_{L^\infty\left(L^\infty\right)}
	\|\Phi\|_{L^\infty((L^2)^2)},
	\end{align*}
	which immediately implies that
	\begin{align}\label{eq:S2_R1}
	\|R_1^n(x)\|_{L^2} &\lesssim\left (\|\lambda_1^n(x)\|_{L^\infty}+\|\lambda_2^n(x)\|_{L^\infty}\right)\|\Phi(t_n)\|_{L^2}\lesssim \tau^3, \\
	\|R_2^n(x)\|_{L^2} &\lesssim \tau^2\|V(t,x)\|_{L^\infty(L^\infty)}^2 \|\lambda_3^n(w,x)\|_{L^\infty([0,\tau];L^2)}\lesssim \tau^3.\label{eq:S2_R2}
	\end{align}
	In view of \eqref{eq:dec:ev},  recalling the definitions of $f_2^n(s)$ and $g_j^n(s,w)$ ($j=2,3,4$) given in Lemma \ref{lemma:strang}, we introduce $f_1^n(s)$ and $g_1^n(s,w)$ such that
	\be\label{eq:S2_fg}
	f^n(s)=f_1^n(s) + f_2^n(s), \quad g^n(s, w) = \sum_{j = 1}^4g_j^n(s, w)
	\ee
	where
	\begin{align*}
	f_1^n(s) &= e^{is\mathcal{D}^\eps}\;\Pi_+^\eps\left(V(t_n+s)e^{-is\mathcal{D}^\eps}\Pi_+^\eps\Phi(t_n)\right)\\
	&\qquad+
	e^{-is\mathcal{D}^\eps}\;\Pi_-^\eps\left(V(t_n+s)e^{is\mathcal{D}^\eps}\Pi_-^\eps\Phi(t_n)\right),\\
	g_1^n(s, w) &= e^{is\mathcal{D}^\eps}\Pi_+^\eps\left(V(t_n)e^{-i(s - w)\mathcal{D}^\eps}\Pi_+^\eps\left(V(t_n)e^{-iw\mathcal{D}^
		\eps}\Pi_+^\eps\Phi(t_n)\right)\right)\\
	& \qquad+ e^{-is\mathcal{D}^\eps}\Pi_-^\eps\left(V(t_n)e^{i(s - w)\mathcal{D}^\eps}\Pi_-^\eps\left(V(t_n)e^{iw\mathcal{D}^
		\eps}\Pi_-^\eps\Phi(t_n)\right)\right).
	\end{align*}
	Denote
	\begin{align*}
	\zeta_{1}^n(x) &= -ie^{-\frac{i\tau}{\eps^2}\mathcal{T}^\eps}\left(\int_0^\tau f_1^n(s)\,ds-\tau f_1^n(\tau/2)\right),\\
	\zeta_{2}^n(x) &= -e^{-\frac{i\tau}{\eps^2}\mathcal{T}^\eps}\left(\int_0^\tau\int_0^sg_1^n(s,w)\,dwds-\frac{\tau^2}{2}g_1^n(\tau/2,\tau/2)\right),
	\end{align*}
	then it is easy to show that  for $J=[0,\tau]^2$,	
	\begin{align}\label{eq:S2_zeta}
	&\|\zeta_{1}^n(x)\|_{L^2}\lesssim \tau^3\|\partial_{ss}f_1(s)\|_{L^\infty([0,\tau];(L^2)^2)}\lesssim\tau^3,\\
	& \|\zeta_{2}^n(x)\|_{L^2} \lesssim   \tau^3(\|\partial_sg_1(s,w)\|_{L^\infty(J;(L^2)^2)}+\|\partial_wg_1(s,w)\|_{L^2(J;(L^2)^2)})\lesssim\tau^3,\label{eq:S2_zeta:2}
	\end{align}
	by noticing that $V\in L^\infty(W^{2m,\infty})$ and $\Phi(t,x)\in L^\infty((H^{2m})^2)$ with $m=2$
	as well as the fact that $\mathcal{D}^\eps:(H^{l})^2\to (H^{l-2})^2$ ($l\ge2$) is uniformly bounded w.r.t. $\eps$.
	Recalling \eqref{eq:fn:S2}, \eqref{eq:gn:S2}, \eqref{eq:eta-S2:1}, \eqref{eq:S2_fg} and $\eta_j^n$ ($j=2,3$) \eqref{eq:eta2:lemma}-\eqref{eq:eta3:lemma} given in Lemma \ref{lemma:strang}, we have
	\be
	\eta^n(x) = \eta_1^n(x) + \eta_2^n(x) + \eta_3^n(x),
	\ee
	where $\eta^n_2(x)$ and $\eta^n_3(x)$ are given in Lemma \ref{lemma:strang}, and
	\begin{align*}
	\eta_1^n(x) = R_1^n(x) + R_2^n(x) + \zeta_{1}^n(x) + \zeta_{2}^n(x).
	\end{align*}
	Combining \eqref{eq:S2_R1}, \eqref{eq:S2_R2}, \eqref{eq:S2_zeta} and \eqref{eq:S2_zeta:2}, we can get
	\be
	\|\eta_1^n(x)\|_{L^2} \leq \|R_1^n(x)\|_{L^2} + \|R_2^n(x)\|_{L^2}  + \|\zeta_{1}^n(x)\|_{L^2} + \|\zeta_{2}^n(x)\|_{L^2} \lesssim\tau^3,
	\ee
	which completes the proof.
\end{proof}

Utilizing these lemmas, we now proceed to prove Theorems \ref{thm:lie} and \ref{thm:strang}.

\bigskip

{\bf Proof of Theorem \ref{thm:lie}}
\begin{proof}
	From Lemma \ref{lemma:lie}, it is straightforward that
	\be\label{eq:rec:1}
	\|{\bf e}^{n+1}(x)\|_{L^2}\leq \|{\bf e}^n(x)\|_{L^2}+\|\eta_1^n(x)\|_{L^2} + \|\eta_2^n(x)\|_{L^2}, \quad 0\leq n\leq\frac{T}{\tau} - 1,
	\ee
	with  ${\bf e}^0(x)=0$, $\|\eta_1^n(x)\|_{L^2}\lesssim\tau^2$ and
	$\eta_2^n(x) = -ie^{-i\tau\mathcal{T}^\eps/\eps^2}\left(\int_0^\tau f_2^n(s)ds - \tau f_2^n(0)\right)$, where $f_2^n(s)$ is defined in \eqref{eq:f2nlie}.
	
	To  analyze  $f_2^n(s)$, using \eqref{eq:pi+} and \eqref{eq:pi-}, we expand $\Pi_+^\eps V(t_n)\Pi_-^\eps$ and $\Pi_-^\eps V(t_n)\Pi_+^\eps$ to get
	\begin{align*}
	\Pi_+^\eps V(t_n)\Pi_-^\eps=&-\eps\Pi_+^0V(t_n)\mathcal{R}_1+\eps\mathcal{R}_1V(t_n)\Pi_-^\eps,\\
	\Pi_-^\eps V(t_n)\Pi_+^\eps=&\eps\Pi_-^0V(t_n)\mathcal{R}_1-\eps\mathcal{R}_1V(t_n)\Pi_+^\eps.
	\end{align*}
	As $\mathcal{R}_1:(H^{m})^2\to (H^{m-1})^2$ is uniformly bounded  with respect to $\eps\in(0,1]$, we have
	\begin{align}\label{eq:pi+eps}
	\left\|\Pi_+^\eps\left(V(t_n)\Pi_-^\eps e^{is\mathcal{D}^\eps}\Phi(t_n)\right)\right\|_{L^2}&\lesssim\eps  \|V(t_n)\|_{W^{1,\infty}}\|\Phi(t_n)\|_{H^1}, \\
	\left\|\Pi_-^\eps\left(V(t_n)\Pi_+^\eps e^{is\mathcal{D}^\eps}\Phi(t_n)\right)\right\|_{L^2}&\lesssim \eps  \|V(t_n)\|_{W^{1,\infty}}\|\Phi(t_n)\|_{H^1}.\label{eq:pi-eps}
	\end{align}
	Noticing the assumptions (A) and (B) with $m=1$ and $m_*=0$, we obtain from \eqref{eq:f2nlie} ($0\leq s\leq\tau$)
	\be\label{eq:S1_ineq}
	\|f_2^n(s)\|_{L^\infty([0,\tau];(L^2)^2)}\lesssim \eps, \quad \|\partial_s(f_2^n)(\cdot)\|_{L^\infty([0,\tau];(L^2)^2)} \lesssim \eps/\eps^2 = 1/\eps.
	\ee
	As a result, from the first inequality, we get
	\be\label{eq:S1_e1}
	\left\|\int_0^\tau f_2^n(s)\,ds-\tau f_2^n(0)\right\|_{L^2}\lesssim \tau\eps.
	\ee
	On the other hand, noticing Taylor expansion and the second inequality in \eqref{eq:S1_ineq}, we have
	\be\label{eq:S1_e2}
	\left\|\int_0^\tau f_2^n(s)\,ds-\tau f_2^n(0)\right\|_{L^2}\leq
	\frac{\tau^2}{2}\|\partial_sf_2^n(\cdot)\|_{L^\infty([0,\tau];(L^2)^2)} \lesssim \tau^2/\eps.	\ee
	Combining \eqref{eq:S1_e1} and \eqref{eq:S1_e2}, we arrive at
	\be
	\|\eta_2^n(x)\|_{L^2}\lesssim \min\{\tau\eps, \tau^2/\eps\}.
	\ee
	Then from \eqref{eq:rec:1} and ${\bf e}^0=0$,  we get
	\begin{align*}
	\|{\bf e}^{n+1}(x)\|_{L^2} \leq & \|{\bf e}^0(x)\|_{L^2} + \sum_{k = 0}^n\|\eta_1^k(x)\|_{L^2} + \sum_{k = 0}^n\|\eta_2^k(x)\|_{L^2}\\
	\lesssim & n\tau^2 + n\min\{\tau\eps, \tau^2/\eps\} \lesssim \tau + \min\{\eps, \tau/\eps\}, \quad 0\leq n\leq\frac{T}{\tau} - 1,
	\end{align*}
	which gives the desired results.
\end{proof}
\smallskip

{\bf Proof of Theorem \ref{thm:strang}}
\begin{proof}
	From Lemma \ref{lemma:strang}, it is easy to get that
	\be\label{eq:rec:2}
	\|{\bf e}^{n+1}(x)\|_{L^2}\leq \|{\bf e}^n(x)\|_{L^2}+\|\eta_1^n(x)\|_{L^2} + \|\eta_2^n(x)\|_{L^2} + \|\eta_3^n(x)\|_{L^2},
	\ee
	with  ${\bf e}^0(x)=0$ and $\|\eta_1^n(x)\|_{L^2}\lesssim \tau^3$.\\
	Through similar computations in the $S_1$ case, under the hypothesis of Theorem \ref{thm:strang}, we can show that for $0\leq s, w,\leq\tau$,
	\begin{align*}
	&\|f_2^n(s)\|_{L^2}\lesssim \eps, \quad \|\partial_sf_2^n(s)\|_{L^2} \lesssim \eps/\eps^2 = 1/\eps, \quad \|\partial_{ss}f_2^n(s)\|_{L^2}\lesssim 1 / \eps^3;\\
	&\|g_j^n(s, w)\|_{L^2}\lesssim\eps, \quad \|\partial_sg_j^n(s, w)\|_{L^2}\lesssim 1 / \eps, \quad \|\partial_wg_j^n(s, w)\|_{L^2}\lesssim 1 / \eps, \quad j = 2, 3, 4.
	\end{align*}
	As a result, for $ j = 2, 3, 4$, we have
	\be\nonumber
	\left\|\int_0^\tau f_2^n(s)\,ds-\tau f_2^n(\frac{\tau}{2})\right\|_{L^2} \lesssim \tau\eps, \; \left\|\int_0^\tau\int_0^sg_j^n(s,w)\,dwds-\frac{\tau^2}{2}g_j^n(\frac{\tau}{2},\frac{\tau}{2})\right\|_{L^2} \lesssim \tau^2\eps.
	\ee
	On the other hand, for $ j = 2, 3, 4$,Taylor expansion will lead to
	\be\nonumber
	\left\|\int_0^\tau f_2^n(s)\,ds-\tau f_2^n(\frac{\tau}{2})\right\|_{L^2} \lesssim \frac{\tau^3}{\eps^3}, \; \left\|\int_0^\tau\int_0^sg_j^n(s,w)\,dwds-\frac{\tau^2}{2}g_j^n(\frac{\tau}{2},\frac{\tau}{2})\right\|_{L^2} \lesssim \frac{\tau^3}{\eps}.
	\ee
	The two estimates above together with \eqref{eq:eta2:lemma} and \eqref{eq:eta3:lemma} imply
	\be
	\|\eta_2^n(x)\|_{L^2} + \|\eta_3^n(x)\|_{L^2}\lesssim\min\{\tau\eps, \tau^3/\eps^3\}.
	\ee
	Recalling \eqref{eq:rec:2}, we can get
	\begin{align*}
	\|{\bf e}^{n+1}(x)\|_{L^2} \leq & \|{\bf e}^0(x)\|_{L^2} + \sum_{k = 0}^n\|\eta_1^k(x)\|_{L^2} + \sum_{k = 0}^n\|\eta_2^k(x)\|_{L^2} + \sum_{k = 0}^n\|\eta_3^k(x)\|_{L^2}\\
	\lesssim & n\tau^3 + n\min\{\tau\eps, \tau^3/\eps^3\} \lesssim \tau^2 + \min\{\eps, \tau^2/\eps^3\}, \quad 0\leq n\leq\frac{T}{\tau} - 1,
	\end{align*}
	which gives the desired results.
\end{proof}

\section{Proof of Theorems \ref{thm:lie2} and \ref{thm:strang2}} \label{sec:nonres}
If  the time step size $\tau$ is away from the resonance, i.e. for given $\eps$, there is a $\delta>0$, such that $\tau\in\mathcal{A}_\delta(\eps)$,  we can show improved uniform error bounds for the splitting methods given in Theorems \ref{thm:lie2} \& \ref{thm:strang2} from Lemmas \ref{lemma:lie} \& \ref{lemma:strang},  as observed in our extensive numerical tests.
\smallskip

{\bf Proof of Theorem \ref{thm:lie2}}
\begin{proof}
	We divide the proof into three steps.
	
	\textbf{Step 1} (Explicit representation of the error). From Lemma \ref{lemma:lie}, we have
	\begin{equation}
	{\bf e}^{n + 1}(x) = e^{-\frac{i\tau}{\eps^2}\mathcal{T}^\eps}e^{-i\int_{t_n}^{t_{n + 1}}V(s)ds}{\bf e}^n(x) + \eta_1^{n}(x) + \eta_2^n(x), \quad 0\leq n\leq \frac{T}{\tau} - 1,
	\end{equation}
	with $\|\eta_1^n(x)\|_{L^2}\lesssim\tau^2$, ${\bf e}^0=0$, $\eta_2^n(x) = -ie^{-\frac{i\tau}{\eps^2}\mathcal{T}^\eps}\left(\int_0^\tau f_2^n(s)ds - \tau f_2^n(0)\right)$ and $f_2^n$ is given in Lemma \ref{lemma:lie} \eqref{eq:f2nlie}.

	Denote the numerical solution propagator  $S_{n, \tau} := e^{-\frac{i\tau}{\eps^2}\mathcal{T}^\eps}e^{-i\int_{t_n}^{t_{n+1}}V(s, x)ds}$ for $n\geq 0$, then
	$\forall \widetilde{\Phi}\in\mathbb{C}^2$, for $m\ge1$,
	\be\label{eq:numint}
	\left\|S_{n,\tau}\widetilde{\Phi}\right\|_{L^2}=\left\|\widetilde{\Phi}\right\|_{L^2},\;
	\left\|S_{n,\tau}\widetilde{\Phi}\right\|_{H^m}
	\leq e^{C\tau\|V(t,x)\|_{L^\infty([0,T];W^{m,\infty})}}\|\widetilde{\Phi}\|_{H^m},
	\ee
	with some generic constant $C>0$
	and
	\begin{align}
	{\bf e}^{n+1}(x) &= S_{n, \tau}{\bf e}^n(x) + \left(\eta_1^n(x) + \eta_2^n(x)\right)\nonumber\\
	&= S_{n, \tau}(S_{n - 1, \tau}{\bf e}^{n - 1}(x)) + S_{n, \tau}\left(\eta_1^{n - 1}(x) + \eta_2^{n - 1}(x)\right) + \left(\eta_1^n(x) + \eta_2^n(x)\right)\nonumber\\
	&= ...\nonumber\\
	&= S_{n, \tau}S_{n - 1, \tau}...S_{0, \tau}{\bf e}^0(x) + \sum_{k = 0}^n S_{n, \tau}...S_{k+2, \tau}S_{k+1, \tau}\left(\eta_1^k(x) + \eta_2^k(x)\right),
	\end{align}
	where for $k = n$, we take $S_{n, \tau}...S_{k+2, \tau}S_{k+1, \tau} = Id$.
	Since $S_{n,\tau}$ preserves the $L^2$ norm, noticing $\|\eta_1^k(x)\|_{L^2}\lesssim \tau^2$, $k = 0, 1, ..., n$, we have
	\begin{equation*}
	\left\|\sum_{k = 0}^nS_{n, \tau}...S_{k+1, \tau}\eta_1^k(x)\right\|_{L^2}\lesssim \sum_{k = 0}^n\tau^2 \lesssim\tau,
	\end{equation*}
	which leads to
	\begin{equation}\label{eq:uniform:s1}
	\|{\bf e}^{n+1}(x)\|_{L^2} \lesssim \tau + \left\|\sum_{k = 0}^n S_{n, \tau}...S_{k+1, \tau}{\eta}_2^k(x)\right\|_{L^2}.
	\end{equation}
	The improved estimates rely on the refined analysis of the terms involving $\eta_2^k$ in \eqref{eq:uniform:s1}. To this aim, we introduce the following approximation of $\eta_2^k$ to focus on the most relevant terms,
	\be\label{eq:eta2t}
	\tilde{\eta}_2^k(x) = \int_0^\tau \tilde{f}_2^k(s)ds - \tau \tilde{f}_2^k(0), \quad k = 0, 1, ..., n,
	\ee
	with
	\be
	\tilde{f}_2^k(s) = -ie^{i(2s-\tau)/\eps^2}\Pi_+^\eps\left(V(t_k)\Pi_-^\eps \Phi(t_k)\right) -ie^{i(\tau-2s)/\eps^2}\Pi_-^\eps\left(V(t_k)\Pi_+^\eps \Phi(t_k)\right),
	\ee
	then it is easy to verify that (using Taylor expansion $e^{i\tau\mathcal{D}^\eps}=Id+O(\tau\mathcal{D}^\eps)$)
	\be
	\|\eta_2^k(x) - \tilde{\eta}_2^k(x)\|_{L^2} \lesssim \tau^2 \|V(t_k)\|_{H^2}\|\Phi(t_k)\|_{H^2}\lesssim\tau^2.
	\ee
	As a result, from \eqref{eq:uniform:s1}, we have
	\begin{align*}
		\|{\bf e}^{n+1}(x)\|_{L^2} \lesssim&\tau+\left\|\sum_{k = 0}^nS_{n, \tau}...S_{k+1, \tau}(\eta_2^k(x) - \tilde{\eta}_2^k(x))\right\|_{L^2}+\left\|\sum_{k = 0}^n S_{n, \tau}...S_{k+1, \tau}\tilde{\eta}_2^k(x)\right\|_{L^2}\\
		\leq & \tau + \sum_{k=0}^n\|\eta_2^k(x) - \tilde{\eta}_2^k(x)\|_{L^2} + \left\|\sum_{k = 0}^n S_{n, \tau}...S_{k+1, \tau}\tilde{\eta}_2^k(x)\right\|_{L^2}\\
		\lesssim &\tau + \left\|\sum_{k = 0}^n S_{n, \tau}...S_{k+1, \tau}\tilde{\eta}_2^k(x)\right\|_{L^2}.
	\end{align*}
	
	\textbf{Step 2} (Representation of the error using the exact solution flow).
	Denote $S_e(t; t_k)$ ($k = 0, 1, ..., n$) to be the exact solution operator of the Dirac equation, acting on some $\tilde{\Phi}(x) = (\tilde{\phi}_1(x), \tilde{\phi}_2(x))^T\in\mathbb{C}^2$ so that $S_e(t; t_k)\tilde{\Phi}(x)$ is the exact solution $\Psi(t, x)$ at time $t$ of
	\begin{equation}\label{eq:exact}
	\left\{
	\begin{aligned}
	& i\partial_t\Psi(t, x) = \frac{\mathcal{T}^\eps}{\eps^2}\Psi(t, x) + V(t, x)\Psi(t, x),\\
	& \Psi(t_k, x) = \tilde{\Phi}(x).
	\end{aligned}\right.
	\end{equation}
	and the following properties hold true for $t\ge t_k$ , $m\ge1$ and some generic constant $C>0$
	\be\label{eq:exaint}
	\left\|S_e(t;t_k)\widetilde{\Phi}\right\|_{L^2}=\left\|\widetilde{\Phi}\right\|_{L^2},\;
	\left\|S_e(t;t_k)\widetilde{\Phi}\right\|_{H^m}
	\leq e^{C(t-t_k)\|V(t,x)\|_{L^\infty([0,T];W^{m,\infty})}}\|\widetilde{\Phi}\|_{H^m}.
	\ee
	It is convenient to write $\tilde{\eta}_2^k(x)$ \eqref{eq:eta2t} as
	\begin{align}\label{eq:S1_eta2}
	\tilde{\eta}_2^k(x)
	=& p_+(\tau)\Pi_+^\eps\left(V(t_k)\Pi_-^\eps\Phi(t_k)\right) +p_-(\tau)\Pi_-^\eps\left(V(t_k)\Pi_+^\eps\Phi(t_k)\right),
	\end{align}
	with $p_\pm(\tau) = -ie^{\mp i\tau/\eps^2}\left(\int_0^\tau e^{\pm i2s/\eps^2}ds - \tau\right)$
	and by the inequality
	$\left|\int_0^\tau e^{i2s/\eps^2}ds - \tau\right| + \left|\int_0^\tau e^{-i2s/\eps^2}ds - \tau\right|\leq 4\tau$ and similar computations in \eqref{eq:pi+eps}-\eqref{eq:pi-eps}, it follows that
	\be\label{eq:smalle}
	\|\tilde{\eta}_2^k\|_{H^2}\lesssim \tau\eps\|V(t_k)\|_{W^{3,\infty}}\|\Phi(t_k)\|_{H^3}\lesssim
	\eps\tau.
	\ee
	Recalling the error bounds in Theorem \ref{thm:lie} and Remark \ref{remark:err}, we have
	\begin{align*}
	\|(S_{n, \tau}...S_{k+1, \tau} - S_e(t_{n+1}; t_{k+1}))\tilde{\eta}_2^k(x)\|_{L^2}&\lesssim \left(\tau+\frac{\tau}{\eps}\right)\|\tilde{\eta}_2^k\|_{H^2}\lesssim\tau^2,
	\end{align*}
	and
	\begin{align}
	\|{\bf e}^{n+1}(x)\|_{L^2\nonumber}
	&\lesssim  \tau + \sum_{k=0}^n\left\|(S_{n, \tau}...S_{k+1, \tau} - S_e(t_{n+1}; t_{k+1}))\tilde{\eta}_2^k(x)\right\|_{L^2}\nonumber \\
	&\qquad+ \left\|\sum_{k = 0}^n S_e(t_{n+1}; t_{k+1})\tilde{\eta}_2^k(x)\right\|_{L^2}\nonumber\\
	&\lesssim  \tau + \left\|\sum_{k = 0}^n S_e(t_{n+1}; t_{k+1})\tilde{\eta}_2^k(x)\right\|_{L^2}.\label{eq:error:lie}
	\end{align}
	Noticing \eqref{eq:S1_eta2}, we have
	\begin{align}
	\label{decomp}
	S_e(t_{n+1}; t_{k+1})\tilde{\eta}_2^k(x)&=  p_+(\tau)S_e(t_{n+1}; t_{k+1})\Pi_+^\eps V(t_k)\Pi_-^\eps S_e(t_k; t_0)\Phi(0)\nonumber\\
	&\quad + p_-(\tau)S_e(t_{n+1}; t_{k+1})\Pi_-^\eps V(t_k)\Pi_+^\eps S_e(t_k; t_0)\Phi(0),
	\end{align}
	and it remains to estimate $S_e$ part in \eqref{eq:error:lie}.
	
	\textbf{Step 3} (Improved error bounds for non-resonant time steps).
	From \cite{BCJT}, we know that the exact solution of Dirac equation is structured as follows
	\begin{equation}\label{eq:split_Pauli}
	S_e(t_{n}; t_{k})\tilde{\Phi}(x) = e^{-i(t_{n} - t_{k})/\eps^2}\Psi_+(t, x) + e^{i(t_{n} - t_{k})/ \eps^2}\Psi_-(t, x) + R_k^n\tilde{\Phi}(x),
	\end{equation}
	where $R_k^n:(H^2)^2\to(L^2)^2$ is the residue operator and $\|R_k^n\tilde{\Phi}(x)\|_{L^2}\lesssim \eps^2\|\tilde{\Phi}(x)\|_{H^2}$ ($0\leq k\leq n$), and
	\begin{equation}
	\left\{
	\begin{aligned}
	i\partial_t\Psi_{\pm}(t, x) & = \pm\mathcal{D}^\eps\Psi_{\pm}(t, x) + \Pi_{\pm}^\eps(V(t)\Psi_{\pm}(t, x)),\\
	\Psi_{\pm}(t_k, x) & = \Pi_{\pm}^\eps\tilde{\Phi}(x).
	\end{aligned}
	\right.
	\end{equation}
	Denote $S_e^+(t; t_k)\tilde{\Phi}(x)=\Psi_+(t,x)$, $S_e^-(t; t_k)\tilde{\Phi}(x)=\Psi_-(t,x)$ to be the solution propagator of the above equation for $\Psi_+(t, x)$, $\Psi_-(t, x)$, respectively, and $S_e^\pm$ share the same properties in \eqref{eq:exaint}. Plugging \eqref{eq:split_Pauli} into \eqref{decomp}, we derive
	\begin{align*}
	&\sum_{k = 0}^nS_e(t_{n+1}; t_{k+1})\tilde{\eta}_2^k(x)\\
	= & \sum_{k = 0}^n\sum_{\sigma=\pm}\left(e^{-i\frac{t_{n+1} - t_{k+1}}{\eps^2}}S_e^+(t_{n+1}; t_{k+1}) + e^{i\frac{t_{n+1} - t_{k+1}}{\eps^2}}S_e^-(t_{n+1}; t_{k+1}) + R_{k+1}^{n+1}\right)\\
	&\quad \Pi_\sigma^\eps V(t_k)\Pi_{\sigma^*}^\eps\left(e^{-i\frac{t_{k} - t_{0}}{\eps^2}}S_e^+(t_{k}; t_0) + e^{i\frac{t_{k} - t_{0}}{\eps^2}}S_e^-(t_{k}; t_0) + R_0^k\right)\Phi(0) p_\sigma(\tau) \\
	= &   \underbrace{\sum_{k = 0}^n\sum_{\sigma=\pm} e^{-i\sigma\frac{t_{n+1} - t_{k+1}}{\eps^2}}S_e^{\sigma}(t_{n+1}; t_{k+1})\Pi_{\sigma}^{\eps} V(t_k)\Pi_{\sigma^*}^{\eps} e^{i\sigma\frac{t_{k} - t_{0}}{\eps^2}} S_e^{\sigma^*}(t_{k}; t_0)\Phi(0)p_\sigma(\tau)}_{I_1^n(x)}  \\
	& +  \underbrace{ \sum_{k=0}^n\sum_{\sigma=\pm}\left(R_{k+1}^{n+1}\Pi_{\sigma}^\eps V(t_k)\Pi_{\sigma*}^{\eps} \Phi(t_k)+S_e(t_{n+1};t_{k+1})\Pi_{\sigma}^\eps V(t_k)\Pi_{\sigma*}^{\eps}R_0^k\Phi(0)\right)p_\sigma(\tau)}_{I_2^n(x)}
	\\
	= & I_1^n(x) + I_2^n(x),
	\end{align*}
	where $\sigma^*=+$ if $\sigma=-$ and $\sigma^*=-$ if $\sigma=+$.
	As $|p_\pm(\tau)|=\left|\int_0^\tau e^{\pm2is/\eps^2}ds - \tau\right|\lesssim\tau^2/\eps^2$ by Taylor expansion, we have
	\begin{align*}
	\|I_2^n(x)\|_{L^2}\lesssim \frac{\tau^2}{\eps^2}\sum_{k = 0}^n \left(\eps^2\|V(t_k)\|_{W^{2,\infty}}\|\Phi(t_k)\|_{H^2}+\eps^2\|V(t_k)\|_{L^{\infty}}\|\Phi(t_0)\|_{H^2} \right)\lesssim\tau.
	\end{align*}
	We can rewrite $I_1^n(x)$ as
	\begin{align*}
	I_1^n(x) = &\sum_{k = 0}^n\sum_{\sigma=\pm} e^{-i\sigma\frac{t_{n+1} -2 t_{k}-\tau}{\eps^2}}S_e^{\sigma}(t_{n+1}; t_{k+1})\Pi_{\sigma}^{\eps} V(t_k)\Pi_{\sigma^*}^{\eps} S_e^{\sigma^*}(t_{k}; t_0)\Phi(0)p_\sigma(\tau),\\
	= &p_+(\tau)\sum_{k = 0}^n (\theta_k - \theta_{k - 1})S_e^+(t_{n+1}; t_{k+1})\Pi_+^\eps V(t_k)\Pi_-^\eps S_e^-(t_k; t_0)\Phi(0)\\
	& + p_-(\tau)\sum_{k = 0}^n \overline{(\theta_k - \theta_{k - 1})}S_e^-(t_{n+1}; t_{k+1})\Pi_-^\eps V(t_k)\Pi_+^\eps S_e^+(t_k; t_0)\Phi(0)\\
	= & \gamma_1^n(x) + \gamma_2^n(x),
	\end{align*}
	where $\bar{\theta}$ is the complex conjugate of $\theta$ and for $\quad 0\leq k\leq n$,
	\begin{align}
	&\theta_k = \sum_{l = 0}^ke^{-i(t_{n+1} - 2t_l - \tau) / \eps^2}=\frac{e^{-in\tau/\eps^2}-e^{-i(n-2k-2)\tau/\eps^2}}{1-e^{2i\tau/\eps^2}}, \quad \theta_{-1} = 0,& \\
	&\gamma_1^n(x) = p_+(\tau)\sum_{k = 0}^n (\theta_k - \theta_{k - 1})S_e^+(t_{n+1}; t_{k+1})\Pi_+^\eps V(t_k)\Pi_-^\eps S_e^-(t_k; t_0)\Phi(0),\label{eq:gm1n}\\
	&\gamma_2^n(x) = p_-(\tau)\sum_{k = 0}^n \overline{(\theta_k - \theta_{k - 1})}S_e^-(t_{n+1}; t_{k+1})\Pi_-^\eps V(t_k)\Pi_+^\eps S_e^+(t_k; t_0)\Phi(0).
	\end{align}
	It is easy to check that if $\tau\in\mathcal{A}_\delta(\eps)$, it satisfies $|1 - e^{2i\tau/\eps^2}|=2|\sin(\tau/\eps^2)|\geq 2\delta > 0$, then we have
	\begin{equation*}
	|\theta_k| \leq \frac{1}{\delta}, \quad k = 0, 1, ..., n.
	\end{equation*}
	As a result, noticing $|p_{\pm}(\tau)|\leq2\tau$, we can get
	\begin{align*}
	&\|\gamma_1^n(x)\|_{L^2}\\
	&\leq 2\tau\bigg\|\sum_{k = 0}^{n - 1}\theta_k\big[S_e^+(t_{n+1}; t_{k+1})\Pi_+^\eps V(t_k)\Pi_-^\eps S_e^-(t_k; t_0)- S_e^+(t_{n+1}; t_{k+2})\Pi_+^\eps V(t_{k+1})\\
	&\qquad\Pi_-^\eps S_e^-(t_{k+1}; t_0)\big]\Phi(0)\bigg\|_{L^2}+\tau\|\theta_n S_e^+(t_{n+1}; t_{n+1})\Pi_+^\eps V(t_n)\Pi_-^\eps S_e^-(t_n; t_0)\Phi(0)\|_{L^2}\\
	&\lesssim \tau\sum_{k = 0}^{n - 1}\tau/\delta + \tau/\delta \lesssim_\delta \tau,
	\end{align*}
	where we have used the triangle inequality and properties of the solution flows $S_e^{\pm}$ to deduce that (omitted for brevity as they are standard)
	{\small\begin{align*}
		&\bigg\|\big[S_e^+(t_{n+1}; t_{k+1})\Pi_+^\eps V(t_k)\Pi_-^\eps S_e^-(t_k; t_0)- S_e^+(t_{n+1}; t_{k+2})\Pi_+^\eps V(t_{k+1})\Pi_-^\eps S_e^-(t_{k+1}; t_0)\big]\Phi(0)\bigg\|_{L^2}\\
		&\leq\left\|S_e^+(t_{n+1}; t_{k+1})\Pi_+^\eps \left( (V(t_k)-V(t_{k+1}))\Pi_-^\eps S_e^-(t_k; t_0)\right)\Phi(0)\right\|_{L^2}\\
		&\quad+\left\|S_e^+(t_{n+1}; t_{k+1})\Pi_+^\eps \left( V(t_{k+1})\Pi_-^\eps (S_e^-(t_{k};t_0)-S_e^-(t_{k+1};t_0))\right)\Phi(0)\right\|_{L^2}\\
		&\quad +\left\|\left(S_e^+(t_{n+1}; t_{k+1})-S_e^+(t_{n+1}; t_{k+2})\right) \Pi_+^\eps V(t_{k+1})\Pi_-^\eps S_e^-(t_{k+1}; t_0)\Phi(0)\right\|_{L^2}\\
		&\lesssim \tau\left\|\partial_{t}V\right\|_{L^\infty(L^\infty)}\|\Phi(0)\|_{L^2}+\tau\left\|\partial_tS_e^-(t;t_0)\Phi(0)\right\|_{L^\infty([0,T];(L^2)^2)}\\
		&\qquad
		+\tau\left\|\partial_t\left(S_e^+(t_{n+1};t)\Pi_+^\eps V(t_{k+1})\Pi_-^\eps S_e^-(t_{k+1}; t_0)\Phi(0)\right) \right\|_{L^\infty([t_{k+1},t_{n+1}];(L^2)^2)}\\
		&\lesssim\tau+\tau\left\|\Phi(0)\right\|_{H^2}+\tau\|V(t_{k+1})\|_{W^{2,\infty}}
		\|\Phi(0)\|_{H^2}\lesssim \tau.\end{align*}}
	Similarly, we could get $\|\gamma_2^n(x)\|_{L^2}\lesssim_\delta \tau$ and hence $\|I_1^n(x)\|_{L^2}\lesssim_\delta \tau$. In summary, we have
	\begin{equation*}
	\|{\bf e}^{n+1}(x)\|_{L^2} \lesssim \tau + \|I_1^n(x)\|_{L^2} + \|I_2^n(x)\|_{L^2}  \lesssim_\delta \tau,
	\end{equation*}
	which gives the desired results.
\end{proof}

{\bf Proof of Theorem \ref{thm:strang2}}
\begin{proof}
	We divide the proof into two steps.	
	
	\textbf{Step 1} (Representation of the error using the exact solution flow). From Lemma \ref{lemma:strang}, we have for $0\leq n\leq\frac{T}{\tau} - 1$,
	\be\label{eq:err_Strang}
	{\bf e}^{n+1}(x) = e^{-\frac{i\tau}{2\eps^2}\mathcal{T}^\eps}e^{-i\int_{t_n}^{t_{n+1}}V(s, x)ds}e^{-\frac{i\tau}{2\eps^2}\mathcal{T}^\eps}{\bf e}^n(x) + \eta_1^n(x) + \eta_2^n(x) + \eta_3^n(x),
	\ee
	with $\eta_j^n$ ($j=1,2,3$) stated in Lemma \ref{lemma:strang} as
	\begin{align}
	&\|\eta_1^n(x)\|_{L^2}\lesssim\tau^3,\quad \eta_2^n(x) = -ie^{-\frac{i\tau}{\eps^2}\mathcal{T}^\eps}\left(\int_0^\tau f_2^n(s)ds - \tau f_2^n(\tau/2)\right), \\
	&\eta_3^n(x) = -e^{-\frac{i\tau}{\eps^2}\mathcal{T}^\eps}\left(\int_0^\tau\int_0^s \sum_{j = 2}^4g_j^n(s, w)dwds - \frac{\tau^2}{2}\sum_{j = 2}^4g_j^n(\tau/2, \tau/2)\right),
	\end{align}
	where $f_2^n$ and $g_j^n$ ($j=2,3,4$) are given in \eqref{eq:f2ns}-\eqref{eq:g4ns}.
	
	Denote the second order splitting integrator $S_{n, \tau} = e^{-\frac{i\tau}{2\eps^2}\mathcal{T}^\eps}e^{-i\int_{t_n}^{t_{n+1}}V(s)ds}e^{-\frac{i\tau}{2\eps^2}\mathcal{T}^\eps}$ for $n
	\geq 0$, and $S_e(t; t_k)$ to be the exact solution flow \eqref{eq:exact} for the Dirac equation  \eqref{eq:dirac_nomag},
	then $S_{n,\tau}$ enjoys the similar properties as those in the first order Lie-Trotter splitting case \eqref{eq:numint} and we can get
	\begin{align}
	{\bf e}^{n+1}(x) =& S_e(t_{n+1}; t_n){\bf e}^n(x) + \eta_1^n(x) + \eta_2^n(x) + \eta_3^n(x) + \left(S_{n, \tau} - S_e(t_{n+1}; t_n)\right){\bf e}^n(x)\nonumber\\
	=& ...\nonumber\\
	=& S_e(t_{n+1}; t_0){\bf e}^0(x) + \sum_{k=0}^nS_e(t_{n+1}; t_{k+1})( \eta_1^k(x) + \eta_2^k(x) + \eta_3^k(x))\nonumber\\
	& + \sum_{k=0}^nS_e(t_{n+1}; t_{k+1})(S_{k, \tau} - S_e(t_{k+1}; t_k)){\bf e}^k(x).\label{eq:enstrang}
	\end{align}
	By Duhamel's principle, it is straightforward to compute
	\begin{align}\label{eq:lips}
	\left(S_{k, \tau} - S_e(t_{k+1}; t_k)\right)\tilde{\Phi}(x)&=
	e^{-\frac{i\tau}{2\eps^2}\mathcal{T}^\eps}(e^{-i\int_{t_k}^{t_{k+1}}V(s,x)ds}-1)e^{-\frac{i\tau}{2\eps^2}\mathcal{T}^\eps}\nonumber\\&\qquad-i\int_0^\tau e^{-\frac{i(\tau-s) \mathcal{T}^\eps}{\eps^2}}
	V(t_k+s,x)S_e(t_k+s;t_k)\tilde{\Phi}(x)\,ds.
	\end{align}
	Recalling $\|e^{-i\int_{t_k}^{t_{k+1}}V(s,x)ds}-1\|_{L^\infty}\leq \tau\|V(t,x)\|_{L^\infty([t_k,t_{k+1}];L^\infty)}$ and the properties of $S_e(t;t_k)$ \eqref{eq:exaint},
	we obtain from \eqref{eq:lips}
	\begin{align*}
	&\left\|\left(S_{k, \tau} - S_e(t_{k+1}; t_k)\right)\tilde{\Phi}(x)\right\|_{L^2}\\
	&\leq
	\tau\|V(t,x)\|_{L^\infty([t_k,t_{k+1}];L^\infty)}\|\tilde{\Phi}\|_{L^2}
	+\tau\|V(t,x)\|_{L^\infty([t_k,t_{k+1}];L^\infty)}\|\tilde{\Phi}\|_{L^2}
	\lesssim \tau\|\tilde{\Phi}\|_{L^2},
	\end{align*}
	and
	\be\label{eq:stb}
	\|S_e(t_{n+1}; t_{k+1})(S_{k, \tau} - S_e(t_{k+1}; t_k)){\bf e}^k(x)\|_{L^2}\lesssim \tau\|{\bf e}^k(x)\|_{L^2}, \quad k = 0, ..., n.
	\ee
	Noticing $\|{\bf e}^0(x)\|_{L^2} = 0$, combining \eqref{eq:stb} and \eqref{eq:enstrang}, recalling $\|\eta_1^n(x)\|_{L^2}\lesssim\tau^3$, we can control
	\begin{align}
	&\|{\bf e}^{n+1}(x)\|_{L^2}\nonumber\\
	& \leq \sum_{k = 0}^n\|S_e(t_{n+1}; t_{k+1})(S_{k, \tau} - S_e(t_{k+1}; t_k)){\bf e}^k\|_{L^2} + \sum_{j=1}^3\left\|\sum_{k=0}^nS_e(t_{n+1}; t_{k+1})\eta_j^k(x)\right\|_{L^2}\nonumber\\
	&\lesssim \tau^2+\sum_{k = 0}^n\tau\|{\bf e}^k(x)\|_{L^2} +\sum_{j=2}^3 \left\|\sum_{k=0}^nS_e(t_{n+1}; t_{k+1})\eta_j^k(x)\right\|_{L^2}.\label{eq:strangerr1}
	\end{align}
	
	Similar to the Lie-Trotter splitting $S_1$, the key to establish the improved error bounds for non-resonant $\tau$ is to derive refined estimates for the terms involving $\eta_j^k$ ($j=2,3$) in \eqref{eq:strangerr1}. To this purpose, we
	introduce the approximations $\tilde{\eta}_l^k(x)$  of $\eta_l^k(x)$ ($l=2,3$, $k=0,1,\ldots,n$) as
	\begin{equation*}
	\tilde{\eta}_2^k(x) = \int_0^\tau \tilde{f}_2^k(s)ds - \tau \tilde{f}_2^k(\frac{\tau}{2}), \; \tilde{\eta}_3^k(x) = \int_0^\tau\int_0^s \sum_{j = 2}^4\tilde{g}_j^k(s, w)dwds - \frac{\tau^2}{2}\sum_{j = 2}^4\tilde{g}_j^k(\frac{\tau}{2}, \frac{\tau}{2}),
	\end{equation*}
	where we expand $V(t_k+s,x)=V(t_k,x)+s\partial_tV(t_k,x)+O(s^2)$  and $e^{is\mathcal{D}^\eps}=Id+is\mathcal{D}^\eps+O(s^2)$ up to the linear term in $f_2^k(s)$ \eqref{eq:f2ns} and the zeroth order term in $g_j^k(s,w)$ ($j=2,3,4$) \eqref{eq:g2ns}-\eqref{eq:g4ns}, respectively,
	\begin{align*}
	\tilde{f}_2^k(s)&=e^{\frac{i(2s-\tau)}{\eps^2}}\left((s-\tau)\mathcal{D}^\eps\Pi_+^\eps(V(t_k)\Pi_-^\eps\Phi(t_k)) +s \Pi_+^\eps(V(t_k)\mathcal{D}^\eps\Pi_-^\eps\Phi(t_k))\right)\\
	&\quad - e^{\frac{i(\tau-2s)}{\eps^2}}\left((s-\tau)\mathcal{D}^\eps\Pi_-^\eps(V(t_k)\Pi_+^\eps\Phi(t_k)) + s\Pi_-^\eps(V(t_k)\mathcal{D}^\eps\Pi_+^\eps\Phi(t_k))\right)\\
	&\quad-ise^{\frac{i(2s-\tau)}{\eps^2}}\Pi_+^\eps(\partial_tV(t_k)\Pi_-^\eps\Phi(t_k))-ise^{\frac{i(\tau-2s)}{\eps^2}} \Pi_-^\eps(\partial_tV(t_k)\Pi_+^\eps\Phi(t_k))
	\\
	& \quad-ie^{\frac{i(2s-\tau)}{\eps^2}}\Pi_+^\eps(V(t_k)\Pi_-^\eps\Phi(t_k)) -i e^{\frac{i(\tau-2s)}{\eps^2}}\Pi_-^\eps(V(t_k)\Pi_+^\eps\Phi(t_k)),\\
	\tilde{g}_2^k(s, w) &=-i e^{\frac{i(2w-\tau)}{\eps^2}}\Pi_+^\eps\left(V(t_k)\Pi_+^\eps\left(V(t_k)\Pi_-^\eps\Phi(t_k)\right)\right) \\
	&\quad-i e^{\frac{i(\tau-2w)}{\eps^2}}\Pi_-^\eps\left(V(t_k)\Pi_-^\eps\left(V(t_k)\Pi_+^\eps\Phi(t_k)\right)\right),\\
	\tilde{g}_3^k(s, w) &= -ie^{\frac{i(2(s-w)-\tau)}{\eps^2}}\Pi_+^\eps\left(V(t_k)\Pi_-^\eps\left(V(t_k)\Pi_+^\eps\Phi(t_k)\right)\right) \\
	&\quad-i e^{\frac{i(\tau-2(s-w))}{\eps^2}}\Pi_-^\eps\left(V(t_k)\Pi_+^\eps\left(V(t_k)\Pi_-^\eps\Phi(t_k)\right)\right),\\
	\tilde{g}_4^k(s, w)& = -ie^{\frac{i(2s-\tau)}{\eps^2}}\Pi_+^\eps\left(V(t_k)\Pi_-^\eps\left(V(t_k)\Pi_-^\eps\Phi(t_k)\right)\right)\\
	&\quad -ie^{\frac{i(\tau-2s)}{\eps^2}}\Pi_-^\eps\left(V(t_k)\Pi_+^\eps\left(V(t_k)\Pi_+^\eps\Phi(t_k)\right)\right).
	\end{align*}
	Using Taylor expansion in $f_2^k(s)$ \eqref{eq:f2ns} and $g_j^k(s,w)$ ($j=2,3,4$) \eqref{eq:g2ns}-\eqref{eq:g4ns} as well as properties of $\mathcal{D}^\eps$, it is not difficult to check that
	\begin{align*}
	&\|\eta_2^k(x)-\tilde{\eta}_2^k(x)\|_{L^2}\\
	&\lesssim\tau^3\bigg( \|V(t,x)\|_{W^{2,\infty}([0,T];L^\infty)} \|\Phi(t_k)\|_{L^2}
	+\|\partial_tV(t,x)\|_{W^{1,\infty}([0,T];H^2)}\|\|\Phi(t_k)\|_{H^2}\\\
	&\qquad+
	\|V(t,x)\|_{L^{\infty}([0,T];H^4)}\|\Phi(t_k)\|_{H^4}\bigg)\lesssim\tau^3,\\
	&\|\eta_3^k(x)-\tilde{\eta}_3^k(x)\|_{L^2}\lesssim\tau^3\|V(t_n,x)\|_{W^{2,\infty}}^2\|\Phi(t_k)\|_{H^2}\lesssim \tau^3,\end{align*}
	which would yield for $k\leq n\leq\frac{T}{\tau}-1$,
	\begin{align}
	\left\|S_e(t_{n+1}; t_{k+1}){\eta}_2^k(x) - S_e(t_{n+1}; t_{k+1})\tilde{\eta}_2^k(x)\right\|_{L^2}&\lesssim\|\eta_2^k(x)-\tilde{\eta}_2^k(x)\|_{L^2}\lesssim \tau^3,\label{eq:31}\\ \left\|S_e(t_{n+1}; t_{k+1})\eta_3^n(x) - S_e(t_{n+1}; t_{k+1})\tilde{\eta}_3^k(x)\right\|_{L^2}&\lesssim
	\|\eta_3^k(x)-\tilde{\eta}_3^k(x)\|_{L^2}\lesssim \tau^3.\label{eq:32}
	\end{align}
	Plugging the above inequalities \eqref{eq:31}-\eqref{eq:32} into \eqref{eq:strangerr1}, we derive
	\begin{align}
	\nonumber\|{\bf e}^{n+1}(x)\|_{L^2} &\lesssim \tau^2+\sum_{k=0}^n\tau^3 + \sum_{j=2}^3\left\|\sum_{k=0}^nS_e(t_{n+1}; t_{k+1})\tilde{\eta}_j^k(x)\right\|_{L^2} + \sum_{k=0}^n\tau\|{\bf e}^k(x)\|_{L^2}\\
	&\lesssim \tau^2+ \sum_{j=2}^3\left\|\sum_{k=0}^nS_e(t_{n+1}; t_{k+1})\tilde{\eta}_j^k(x)\right\|_{L^2} + \sum_{k=0}^n\tau\|{\bf e}^k(x)\|_{L^2}.\label{eq:err_Strang_split}
	\end{align}
	
	\textbf{Step 2} (Improved estimates for non-resonant time steps).
	It remains to show the estimates on the  terms related to $\tilde{\eta}_2^k$ and $\tilde{\eta}_3^k$. The arguments will be similar to  those in the proof of the Lie-Trotter splitting case Theorem \ref{thm:lie2}, so we only sketch the proof below. Taking $\tilde{\eta}_2^k$ for example, we write
	\be\label{eq:etasplit}
	\tilde{\eta}_{2}^k(s) = \tilde{\eta}_{2+}^k(s) + \tilde{\eta}_{2-}^k(s),\quad
	\tilde{\eta}_{2\pm}^k(x) = \int_0^\tau \tilde{f}_{2\pm}^k(s)ds - \tau \tilde{f}_{2\pm}^k(\tau/2),  \quad k = 0, 1, ..., n,
	\ee
	with
	\begin{align*}
	\tilde{f}_{2\pm}^k(s) = &e^{\pm i(2s-\tau)/\eps^2}\left(\pm(s-\tau)\mathcal{D}^\eps\Pi_{\pm}^\eps(V(t_k)\Pi_{\mp}^\eps\Phi(t_k))  \pm s\Pi_{\pm}^\eps(V(t_k)\mathcal{D}^\eps\Pi_{\mp}^\eps\Phi(t_k))\right)\\
	& -i se^{\pm i(2s-\tau)/\eps^2}\ \Pi_{\pm}^\eps(\partial_tV(t_k)\Pi_{\mp}^\eps\Phi(t_k))-ie^{\pm i(2s-\tau)/\eps^2}\Pi_{\pm}^\eps(V(t_k)\Pi_{\mp}^\eps\Phi(t_k))\end{align*}
	and $\tilde{f}_{2}^k(s) = \tilde{f}_{2+}^n(s) + \tilde{f}_{2-}^n(s)$.

	Recalling the structure of the exact solution to the Dirac equation in \eqref{eq:split_Pauli}, we have for $0\leq k\leq n$
	\begin{equation*}
	S_e(t_{n}; t_{k})\tilde{\Phi}(x) = e^{-i(t_{n} - t_{k})/\eps^2}S_e^+(t_n; t_k)\tilde{\Phi}( x) + e^{i(t_{n} - t_{k})/ \eps^2}S_e^-(t_n; t_k)\tilde{\Phi}(x) + R_k^n\tilde{\Phi}(x),
	\end{equation*}
	where the propagators $S_e^{\pm}$ and the residue operator $R_k^n:(L^2)^2\to (L^2)^2$ are defined in \eqref{eq:split_Pauli}. Therefore, we can get
	\begin{align*}
	\sum_{k=0}^nS_e(t_{n+1}; t_{k+1})\tilde{\eta}_{2+}^k(x) = \sum_{j = 1}^4\tilde{I}_j^n(x),
	\end{align*}
	with
	\begin{align*}
	\tilde{I}_1^n(x) &= \tilde{p}_1(\tau)\sum_{k=0}^ne^{-\frac{i(t_{n+1} - 2t_k - \tau)}{\eps^2}}S_e^+(t_{n+1}; t_{k+1})\Pi_+^\eps V(t_k)\Pi_-^\eps S_e^-(t_k; t_0)\Phi(0),\\
	\tilde{I}_2^n(x) &= \tilde{p}_1(\tau)\sum_{k = 0}^n(R_{k+1}^{n+1}\Pi_+^\eps V(t_k)\Pi_-^\eps\Phi(t_k) + S_e(t_{n+1}; t_{k+1})\Pi_+^\eps V(t_k)\Pi_-^\eps R_0^k\Phi(0)),\\
	\tilde{I}_3^n(x) &= \sum_{k=0}^ne^{-\frac{i(t_{n+1} - 2t_k - \tau)}{\eps^2}}S_e^+(t_{n+1}; t_{k+1})\bigg(\tilde{p}_2(\tau)\mathcal{D}^\eps\Pi_+^\eps V(t_k)\\
	&\hskip2cm+ \tilde{p}_3(\tau)\Pi_+^\eps V(t_k)\mathcal{D}^\eps -i\tilde{p}_3(\tau) \Pi_+^\eps \partial_tV(t_k)\bigg)\Pi_-^\eps S_e^-(t_k; t_0)\Phi(0),\\
	\tilde{I}_4^n(x) &=\sum_{k=0}^n\bigg(R_{k+1}^{n+1}\left(\tilde{p}_2(\tau)\mathcal{D}^\eps\Pi_+^\eps V(t_k) + \tilde{p}_3(\tau)\left(\Pi_+^\eps V(t_k)\mathcal{D}^\eps -i\Pi_+^\eps \partial_tV(t_k)\right)\right)\Pi_-^\eps\Phi(t_k)\\
	&\qquad+ S_e(t_{n+1}; t_{k+1})\left(\tilde{p}_2(\tau)\mathcal{D}^\eps\Pi_+^\eps V(t_k) +\tilde{p}_3(\tau)\left(\Pi_+^\eps V(t_k)\mathcal{D}^\eps -i\Pi_+^\eps \partial_tV(t_k)\right)\right)
	\\ &\hskip2cm\Pi_-^\eps R_0^k\Phi(0)\bigg),
	\end{align*}
	where
	\begin{align*}
	&\tilde{p}_1(\tau)=-i\left(\int_0^\tau e^{i(2s-\tau)/\eps^2}ds - \tau \right),\quad
	\tilde{p}_2(\tau)=\left(\int_0^\tau (s-\tau)e^{i(2s-\tau)/\eps^2}ds +\frac{\tau^2}{2} \right),\\
	&\tilde{p}_3(\tau)=\left(\int_0^\tau se^{i(2s-\tau)/\eps^2}ds - \frac{\tau^2}{2}\right).
	\end{align*}
	The residue terms $\tilde{I}_2^n$ and $\tilde{I}_4^n$ will be estimated first.  Using the properties of $R_k^n$ and $S_e$, noticing \eqref{eq:pi+eps}-\eqref{eq:pi-eps}, we have
	\begin{align*}
	&\|R_{k+1}^{n+1}\Pi_+^\eps V(t_k)\Pi_-^\eps\Phi(t_k) + S_e(t_{n+1}; t_{k+1})\Pi_+^\eps V(t_k)\Pi_-^\eps R_0^k\Phi(0)\|_{L^2} \\
	&\lesssim \eps^3\|V(t_k)\|_{W^{3,\infty}}\left(\|\Phi(t_k)\|_{H^3}+\|\Phi(0)\|_{H^3}\right),\\
	&\|R_{k+1}^{n+1}\mathcal{D}^\eps\Pi_+^\eps V(t_k) \Pi_-^\eps\Phi(t_k)\|_{L^2}+\| R_{k+1}^{n+1} (\Pi_+^\eps V(t_k)\mathcal{D}^\eps -i \Pi_+^\eps \partial_tV(t_k))\Pi_-^\eps\Phi(t_k)\|_{L^2}\\
	&\lesssim \eps^3\|V(t,x)\|_{W^{1,\infty}([0,T];W^{5,\infty})}\|\Phi(t_k)\|_{H^5},\\
	&\|S_e(t_{n+1}; t_{k+1})\mathcal{D}^\eps\Pi_+^\eps V(t_k)\Pi_-^\eps R_0^k\Phi(0)\|_{L^2} \lesssim \eps^3\|V(t,x)\|_{W^{1,\infty}([0,T];W^{3,\infty})}\|\Phi(0)\|_{H^5},\\
 &\|S_e(t_{n+1}; t_{k+1})(\Pi_+^\eps V(t_k)\mathcal{D}^\eps -i \Pi_+^\eps \partial_tV(t_k))\Pi_-^\eps R_0^k\Phi(0)\|_{L^2} \\
	&\lesssim \eps^3\|V(t,x)\|_{W^{1,\infty}([0,T];W^{3,\infty})}\|\Phi(0)\|_{H^5},
	\end{align*}
	which will lead to the following conclusions in view of the fact that
	$|\tilde{p}_1(\tau)|=|\int_0^\tau e^{i(2s-\tau)/\eps^2}ds - \tau | \lesssim \min\{\tau^2 / \eps^2, \tau^3 / \eps^4\}$ and $|\tilde{p}_2(\tau)|,|\tilde{p}_3(\tau)|\lesssim \min\{\tau^2 / \eps^2, \tau^3 / \eps^4\}$ (Taylor expansion up to the linear or the quadratic term),
	\be
	\|\tilde{I}_2^n(x)\|_{L^2} \lesssim\min\{\tau\eps, \tau^2/\eps\}, \quad \|\tilde{I}_4^n(x)\|_{L^2} \lesssim\min\{\tau\eps, \tau^2/\eps\}.
	\ee
	Now, we proceed to treat $\tilde{I}_1^n$ and $\tilde{I}_3^n$.
	For $\tilde{I}_1^n(x)$, it is similar to \eqref{eq:gm1n} which has been analyzed in the $S_1$ case. Using the same idea (details omitted for brevity here), and the fact that  $|\tilde{p}_1(\tau)|=\left|\int_0^\tau e^{i(2s-\tau)/\eps^2}ds - \tau \right| \lesssim \min\{\tau, \tau^2 / \eps^2\}$ as well as $\Pi_{\pm}^\eps V(t_k)\Pi_{\mp}^\eps=O(\eps)$, under the regularity assumptions,
	we can get for $\tau\in\mathcal{A}_\delta(\eps)$,
	\be
	\|\tilde{I}_1^n(x)\|_{L^2} \lesssim \min\{\tau, \tau^2 / \eps^2\}(\sum_{k=0}^{n-1}\tau\eps/\delta + \eps / \delta) \lesssim_\delta \min\{\tau\eps, \tau^2 / \eps\}.
	\ee
	Similarly, noticing $|\tilde{p}_2(\tau)|,|\tilde{p}_3(\tau)|\leq \tau^2$,
	we can get
	\be
	\|\tilde{I}_3^n(x)\|_{L^2} \lesssim \tau^2(\sum_{k=0}^{n-1}\tau\eps/\delta + \eps / \delta) \lesssim_\delta \tau^2\eps.
	\ee
	Combing the estimates for $\tilde{I}_j^n$ ($j=1,2,3,4$), we have
	\be\label{eq:eta+}
	\left\|\sum_{k=0}^nS_e(t_{n+1}; t_{k+1})\tilde{\eta}_{2+}^k(x)\right\|_{L^2} \leq \sum_{j = 1}^4\|\tilde{I}_j^n(x)\|_{L^2} \lesssim_\delta \min\{\tau\eps, \tau^2/\eps\}.
	\ee
	For $\sum_{k=0}^nS_e(t_{n+1}; t_{k+1})\tilde{\eta}_{2-}^k(x)$, we can have the same results as
	\be
	\left\|\sum_{k=0}^nS_e(t_{n+1}; t_{k+1})\tilde{\eta}_{2-}^k(x)\right\|_{L^2} \lesssim_\delta \min\{\tau\eps, \tau^2/\eps \},
	\ee
	which yield the following results in view of \eqref{eq:eta+} and \eqref{eq:etasplit}
	\be
	\left\|\sum_{k=0}^nS_e(t_{n+1}; t_{k+1})\tilde{\eta}_{2}^k(x)\right\|_{L^2} \lesssim_\delta \min\{\tau\eps, \tau^2/\eps\}.
	\ee
	The same technique works for $S_e(t_{n+1}; t_{k+1})\tilde{\eta}_{3}^k(x)$ and we can get
	\be
	\left\|\sum_{k=0}^nS_e(t_{n+1}; t_{k+1})\tilde{\eta}_{3}^k(x)\right\|_{L^2} \lesssim_\delta \min\{\tau\eps, \tau^2/\eps \}.
	\ee
	Plugging these results into \eqref{eq:err_Strang_split}, we have
	\be
	\|{\bf e}^{n+1}(x)\|_{L^2} \lesssim_\delta \tau^2 + \sum_{k=0}^n\tau\|{\bf e}^k(x)\|_{L^2} + \min\{\tau\eps, \tau^2 / \eps\}.
	\ee
	Gronwall's inequality then implies for  $\tau$ satisfying $\tau\in\mathcal{A}_\delta(\eps)$,
	\be
	\|{\bf e}^{n+1}(x)\|_{L^2} \lesssim_\delta \tau^2 + \min\{\tau\eps, \tau^2/\eps\},\quad 0\leq n\leq\frac{T}{\tau}-1.
	\ee
	This completes the proof for Theorem \ref{thm:strang2}.
\end{proof}

\section{Numerical results}
In this section, we report three numerical examples to verify our theorems. For spatial discretization, we use Fourier pseudospectral method.

In the first two 1D examples, we choose the electric potential in \eqref{eq:dirac_nomag} as
\be
V(t, x) = \frac{1 - x}{1 + x^2}, \quad x\in\mathbb{R}, \quad t \geq 0,
\ee
and the initial data in \eqref{eq:dirac_initial} for the first two examples in 1D as
\be
\phi_1(0, x) =  e^{-\frac{x^2}{2}}, \quad \phi_2(0, x) = e^{-\frac{(x - 1)^2}{2}}, \quad x\in\mathbb{R}.
\ee
In the last example, which is a 2D problem, we choose the electric potential in \eqref{eq:dirac4_nomag} as the honey-comb lattice potential with $\bx = (x_1, x_2)^T\in\mathbb{R}^2$
\be
V(t, \bx) = \cos(-\frac{4\pi}{\sqrt{3}}x_1) + \cos(\frac{2\pi}{\sqrt{3}}x_1 + 2\pi x_2) +  \cos(\frac{2\pi}{\sqrt{3}}x_1 - 2\pi x_2),
\ee
and the initial data in \eqref{eq:initial} are chosen as
\begin{align}
&\psi_1(0, \bx) =  e^{-\frac{x_1^2+x_2^2}{2}}, \quad \psi_2(0, \bx) = e^{-\frac{(x_1 - 1)^2+x_2^2}{2}},\\
&\psi_3(0, \bx) =  e^{-\frac{(x_1+1)^2 + (x_2+1)^2}{2}}, \quad \psi_4(0, \bx) = e^{-\frac{x_1^2 + (x_2-1)^2}{2}}.
\end{align}

In the 1D numerical simulations, as a common practice, we truncate the whole space onto a sufficiently large bounded domain $\Omega = (a, b)$, and assume periodic boundary conditions. The mesh size is chosen as $h := \triangle x = \frac{b - a}{M}$ with $M$ being an even positive integer. Then the grid points can be denoted as $x_j := a + jh$, for $j = 0, 1, ..., M$.

To show the numerical results, we introduce the discrete $l^2$ errors of the numerical solution. Let $\Phi^n = (\Phi_0^n, \Phi_1^n, ..., \Phi_{M-1}^n,\Phi_M^n)^T$ be the numerical solution obtained by a numerical method with time step $\tau$ and  $\eps$ as well as a very fine mesh size
$h$ at time $t = t_n$, and $\Phi(t, x)$ be the exact solution, then the relative discrete $l^2$ error is quantified as
\be
e^{\eps, \tau}(t_n) = \frac{\|\Phi^n - \Phi(t_n, \cdot)\|_{l^2}}{\|\Phi(t_n, \cdot)\|_{l^2}} = \frac{\sqrt{h\sum_{j = 0}^{M - 1}|\Phi_j^n - \Phi(t_n, x_j)|^2}}{\sqrt{h\sum_{j=0}^{M-1}|\Phi(t_n, x_j)|^2}},
\ee
and $e^{\eps, \tau}(t_n)$ should be close to the $L^2$ errors (with normalized probability density of the wave function) in Theorems \ref{thm:lie}, \ref{thm:strang}, \ref{thm:lie2} \& \ref{thm:strang2} for fine spatial mesh sizes $h$.\\
For the 2D example, with similar notations (equal mesh size $h$ and grid points along each direction), the relative discrete $l^2$ error could be defined as
\be
e^{\eps, \tau}(t_n) = \frac{\|\Psi^n - \Psi(t_n, \cdot)\|_{l^2}}{\|\Psi(t_n, \cdot)\|_{l^2}} = \frac{h\sqrt{\sum_{j = 0}^{M^2 - 1}|\Psi_j^n - \Psi(t_n, \bx_j)|^2}}{h\sqrt{\sum_{j=0}^{M^2-1}|\Psi(t_n, \bx_j)|^2}}.
\ee

\noindent\textbf{Example 1} We first test the uniform error bounds for the splitting methods. In this example, we choose resonant time step size, that is, for small enough chosen $\eps$, there is a positive $k_0$, such that $\tau = k_0\eps\pi$.

The bounded computational domain is set as $\Omega = (-32, 32)$.
Because we are only concerned with the temporal errors in this paper, during the computation, the spatial mesh size is always set to be $h = \frac{1}{16}$ so that the spatial error is negligible.
As there is no  exact solution available, for comparison, we use a numerical `exact' solution generated by  the $S_2$ method with a very fine time step size $\tau_e = 2\pi\times10^{-6}$.

Tables \ref{table:LDirac_S1} \& \ref{table:LDirac_S2} show the numerical errors $e^{\eps, \tau}(t = 2\pi)$ with different $\varepsilon$ and time step size $\tau$ for $S_1$ and $S_2$, respectively. \\

\begin{table}[htp]
	\def\temptablewidth{1\textwidth}
	\caption{Discrete $l^2$ temporal errors $e^{\eps, \tau}(t = 2\pi)$ for the wave function with resonant time step size, $S_1$ method. }
	{\rule{\temptablewidth}{1pt}}
	\begin{tabular*}{\temptablewidth}{@{\extracolsep{\fill}}ccccccc}
		$e^{\eps, \tau}(t = 2\pi)$  & $\tau_0 = \pi/4$ & $\tau_0 / 4$ & $\tau_0 / 4^2$ & $\tau_0 / 4^3$
		& $\tau_0 / 4^4$ & $\tau_0 / 4^5$ \\ \hline
		$\eps_0=1$ & 4.84E-1 &	\textbf{1.27E-1} & 3.20E-2 &	8.03E-3 & 2.01E-3 &	5.02E-4\\
		order & -- & \textbf{0.97} &	0.99 &	1.00 &	1.00 &	1.00\\\hline
		$\eps_0/2$ & 6.79E-1 &	1.21E-1 &	\textbf{3.10E-2} &	7.78E-3 &	1.95E-3 &	4.87E-4\\
		order & -- & 1.24 & \textbf{0.98} &	1.00 &	1.00 &	1.00\\\hline
		$\eps_0/2^2$ & 5.78E-1 & 2.71E-1 &	3.07E-2 &	\textbf{7.76E-3} &	1.95E-3 &	4.87E-4\\
		order & -- & 0.55 &	1.57 &	\textbf{0.99} &	1.00 &	1.00\\\hline
		$\eps_0/2^3$ & \textbf{5.33E-1} & 1.85E-1 &	1.21E-1 &	7.75E-3 &	\textbf{1.95E-3} &	4.87E-4\\
		order & \textbf{--} & 0.76 &	0.30 &	1.98 &	\textbf{1.00} &	1.00\\\hline
		$\eps_0/2^4$ & 5.13E-1 & 1.48E-1 &	7.02E-2 &	5.76E-2 &	1.95E-3 &	\textbf{4.88E-4}\\
		order & -- & 0.90 &	0.54 &	0.14 &	2.44  & \textbf{1.00}\\\hline
		$\eps_0/2^5$ & 5.04E-1 &	\textbf{1.34E-1} &	4.70E-2 & 3.07E-2 &	2.82E-2 &	4.88E-4\\
		order & -- & \textbf{0.96} & 0.75 &	0.31 &	0.06 &	2.93\\\hline
		$\eps_0/2^7$ & 4.98E-1 & 1.25E-1 &	\textbf{3.37E-2} &	1.18E-2 &	7.68E-3 &	7.05E-3\\
		order & -- & 1.00 &	\textbf{0.95} &	0.76 &	0.31 &	0.06\\\hline
		$\eps_0/2^9$ & 4.97E-1 & 1.24E-1 &	3.17E-2 &	\textbf{8.46E-3} &	2.95E-3 &	1.92E-3\\
		order & -- & 1.00 &	0.98 &	\textbf{0.95} &	0.76 &	0.31\\\hline
		$\eps_0/2^{11}$ & 4.96E-1 &	1.23E-1 &	3.13E-2 &	7.94E-3 & \textbf{2.12E-3} &	7.37E-4\\
		order & -- & 1.00 &	0.99 &	0.99 &	\textbf{0.95} &	0.76 \\\hline\hline
		$\max\limits_{0<\eps\leq 1}e^{\eps, \tau}(t=2\pi)$ & 6.79E-1 &	2.71E-1 &	1.21E-1 &	5.76E-2 &	2.82E-2 &	1.39E-2\\
		order & -- & 0.66 &	0.58 &	0.54 &	0.52 &	0.51
	\end{tabular*}
	{\rule{\temptablewidth}{1pt}}
	\label{table:LDirac_S1}
\end{table}

\begin{table}[htp]
	\def\temptablewidth{1\textwidth}
	\caption{Discrete $l^2$ temporal errors $e^{\eps, \tau}(t = 2\pi)$ for the wave function with resonant time step size, $S_2$ method. }
	{\rule{\temptablewidth}{1pt}}
	
	\begin{tabular*}{\temptablewidth}{@{\extracolsep{\fill}}ccccccc}
		$e^{\eps, \tau}(t = 2\pi)$	 & $\tau_0 = \pi / 4$ & $\tau_0 / 4^2$ & $\tau_0 / 4^3$ & $\tau_0 / 4^4$
		& $\tau_0 / 4^5$ & $\tau_0 / 4^6$ \\ \hline
		$\eps_0 = 1$ & 8.08E-2 &	\textbf{4.44E-3} &	2.76E-4 &	1.73E-5 &	1.08E-6 &	6.74E-8\\
		order & -- & \textbf{2.09} &	2.00 &	2.00 &	2.00 &	2.00\\\hline
		$\eps_0 / 2$ & 4.13E-1 &	9.66E-3 &	\textbf{5.73E-4}	& 3.57E-5 &	2.23E-6 &	1.39E-7\\
		order & -- & 2.71 &	\textbf{2.04} &	2.00 &	2.00 &	2.00\\\hline
		$\eps_0 / 2^2$ & 2.63E-1 &	2.15E-1 &	1.21E-3 &	\textbf{7.22E-5} &	4.50E-6 &	2.81E-7\\
		order & -- & 0.15 &	3.74 &	\textbf{2.03} &	2.00 &	2.00\\\hline
		$\eps_0 / 2^3$ & \textbf{2.08E-1} &	1.10E-1 &	1.10E-1 &	1.51E-4 &	\textbf{9.05E-6} &	5.64E-7\\
		order & \textbf{--} & 0.46 &	0.00 &	4.75 &	\textbf{2.03} &	2.00\\\hline
		$\eps_0 / 2^4$ & 1.92E-1 &	5.56E-2 &	5.51E-2 &	5.51E-2 &	1.89E-5 &	\textbf{1.13E-6}\\
		order & -- & 0.89 &	0.01 &	0.00 &	5.76 &	\textbf{2.03}\\\hline
		$\eps_0 / 2^5$ & 1.88E-1 &	2.85E-2 &	2.76E-2 &	2.76E-2 &	2.76E-2 &	2.36E-6\\
		order & -- & 1.36 &	0.02 &	0.00 &	0.00 &	6.76\\\hline
		$\eps_0 / 2^6$ & 1.87E-1 &	1.55E-2 &	1.38E-2 &	1.38E-2 &	1.38E-2 &	1.38E-2\\
		order & -- & 1.79 &	0.08 &	0.00 &	0.00 &	0.00\\\hline
		$\eps_0 / 2^7$ & 1.87E-1 &	\textbf{9.86E-3} &	6.92E-3 &	6.90E-3 &	6.90E-3 &	6.90E-3\\
		order & -- & \textbf{2.12} &	0.26 &	0.00 &	0.00 &	0.00\\\hline
		$\eps_0 / 2^{11}$ & 1.87E-1 &	6.97E-3 &	\textbf{5.93E-4} &	4.32E-4 &	4.31E-4 &	4.31E-4\\
		order & -- & 2.37 &	\textbf{1.78} &	0.23 &	0.00 &	0.00\\\hline
		$\eps_0 / 2^{15}$ & 1.87E-1 &	6.95E-3 &	4.03E-4 &	\textbf{3.75E-5} &	2.71E-5 &	2.70E-5\\
		order & -- & 2.37 &	2.05 &	\textbf{1.71} &	0.23 &	0.00\\\hline\hline
		$\max\limits_{0<\eps\leq 1}e^{\eps, \tau}(t = 2\pi)$& 4.13E-1 &	2.15E-1 &	1.10E-1 &	5.51E-2 &	2.76E-2 &	1.38E-2\\
		order & -- & 0.47 &	0.49 &	0.50 &	0.50 &	0.50	
	\end{tabular*}
	{\rule{\temptablewidth}{1pt}}
	\label{table:LDirac_S2}
\end{table}

\begin{figure}[htp]
	\vspace{5pt}
	\caption{Order plot for $S_2$ using resonant time steps.}
	\includegraphics[width=0.6\textwidth]{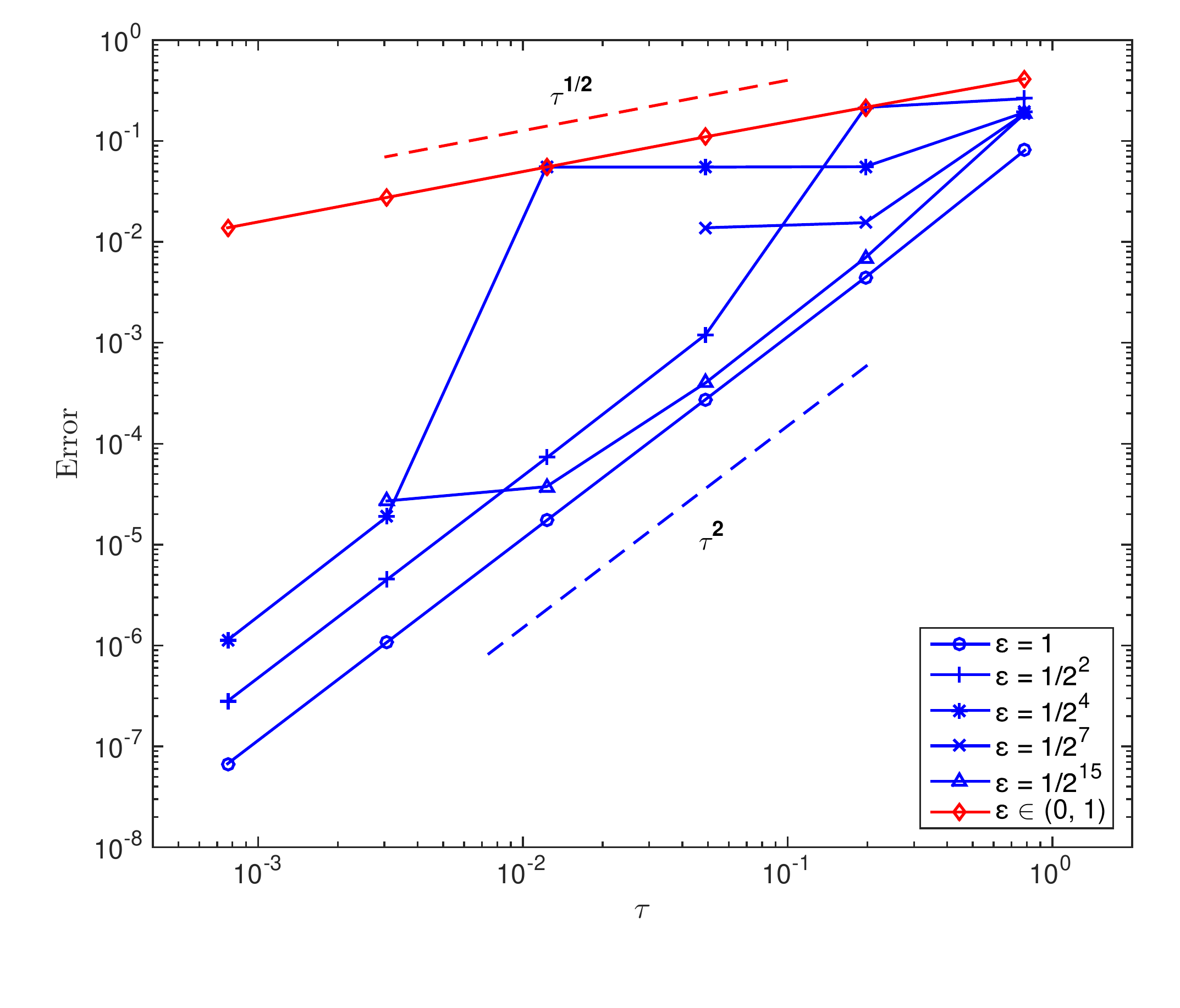}
	\label{fig:S2_res}
\end{figure}

In Tables \ref{table:LDirac_S1} \& \ref{table:LDirac_S2}, the last two rows show the largest error of each column for fixed $\tau$. They both give $1/2$ order of convergence, which coincides well with Theorems \ref{thm:lie} \& \ref{thm:strang}.
More specifically, in Table \ref{table:LDirac_S1}, we can see when $\tau\gtrsim \eps$ (below the lower bolded line), there is first order convergence, which agrees with the error bound $\|\Phi(t_n, x) - \Phi^n(x)\|_{L^2}\lesssim \tau +\eps$. When $\tau\lesssim\eps^2$ (above the upper bolded line), we  observe first order convergence, which matches the other error bound $\|\Phi(t_n, x) - \Phi^n(x)\|_{L^2}\lesssim \tau +\tau/ \eps$. Similarly, in Table \ref{table:LDirac_S2}, the second order convergence can be clearly observed when $\tau\lesssim\eps^2$ (above the upper bolded line) or when $\tau\gtrsim\sqrt{\eps}$ (below the lower bolded line), which fits well with the two error bounds $\|\Phi(t_n,x)-\Phi^n(x)\|_{L^2}\lesssim \tau^2+\tau^2/\eps^3$ and $\|\Phi(t_n,x)-\Phi^n(x)\|_{L^2}\lesssim \tau^2+\eps$.

Moreover, Figure \ref{fig:S2_res} gives the order plot for $S_2$ under resonant time steps. It could be clearly observed that when $\eps$ is relatively large, there is second order convergence for small  time step sizes; and when $\eps$ is relatively small, there is second order convergence for large time step sizes. Overall, there is a $1/2$ order uniform convergence, which corresponds well with Theorem \ref{thm:strang}.

Through the results of this example, we successfully validate the uniform error bounds for the splitting methods in Theorems \ref{thm:lie} \& \ref{thm:strang}.\\

\noindent\textbf{Example 2} In this example, we test the improved uniform error bounds for  non-resonant time step size. Here we choose $\tau\in\mathcal{A}_\delta(\eps)$ for  some given $\eps$ and  $0<\delta \leq1$.

The bounded computational domain is set as $\Omega = (-16, 16)$.
The numerical `exact' solution is computed by  the $S_2$ method with a very small time step $\tau_e = 8\times10^{-6}$. Spatial mesh size is fixed as $h=1/16$ for all the numerical simulations.

Tables \ref{table:LDirac_S1_nrs} \& \ref{table:LDirac_S2_nrs} show the numerical errors $e^{\eps, \tau}(t = 4)$ with different $\varepsilon$ and time step size $\tau$ for $S_1$ and $S_2$, respectively.

\begin{table}[htp]
	\def\temptablewidth{1\textwidth}
	\caption{Discrete $l^2$ temporal errors $e^{\eps, \tau}(t = 4)$ for the wave function with non-resonant time step size, $S_1$ method. }
	{\rule{\temptablewidth}{1pt}}
	
	\begin{tabular*}{\temptablewidth}{@{\extracolsep{\fill}}cccccccc}
		$e^{\eps, \tau}(t = 4)$& $\tau_0 = 1/2$ & $\tau_0 / 2$ & $\tau_0 / 2^2$ & $\tau_0 / 2^3$
		& $\tau_0 / 2^4$ & $\tau_0 / 2^5$ & $\tau_0 / 2^6$\\ \hline
		$\eps_0 = 1$ &	3.51E-1 &	1.78E-1 &	8.96E-2 &	4.50E-2 &	2.25E-2 &	1.13E-2 &	5.64E-3\\
		order & -- & 0.98 &	0.99 &	0.99 &	1.00 &	1.00 &	1.00\\\hline
		$\eps_0 / 2$ &	3.52E-1 &	1.65E-1 &	8.34E-2 &	4.20E-2 &	2.11E-2 &	1.05E-2 &	5.28E-3\\
		order & -- &	1.10 &	0.98 &	0.99 &	1.00 &	1.00 &	1.00\\\hline
		$\eps_0 / 2^2$  &	3.25E-1 &	1.64E-1 &	8.04E-2 &	4.07E-2 &	2.05E-2 &	1.03E-2 &	5.15E-3\\
		order & -- & 0.99 &	1.03 &	0.98 &	0.99 &	1.00 &	1.00\\\hline
		$\eps_0 / 2^3$ & 3.24E-1 &	1.69E-1 &	8.10E-2 &	4.13E-2 &	2.02E-2 &	1.02E-2 &	5.13E-3\\
		order & -- &0.94 &	1.06 &	0.97 &	1.03 &	0.99 &	0.99\\\hline
		$\eps_0 / 2^4$ &3.12E-1 &	1.61E-1 &	8.24E-2 &	4.22E-2 &	2.05E-2 &	1.03E-2 &	5.10E-3\\
		order & -- & 0.95 &	0.97 &	0.97 &	1.04 &	0.99 &	1.02\\\hline
		$\eps_0 / 2^5$ & 3.25E-1 &	1.61E-1 &	8.10E-2 &	4.10E-2 &	2.07E-2 &	1.04E-2 &	5.13E-3\\
		order & -- & 1.02 &	0.99 &	0.98 &	0.99 &	0.98 &	1.02\\\hline
		$\eps_0 / 2^6$ & 3.19E-1 &	1.63E-1 &	8.43E-2 &	4.09E-2 &	2.05E-2 &	1.03E-2 &	5.16E-3\\
		order & -- & 	0.97 &	0.95 &	1.04 &	1.00 &	0.99 &	0.99\\\hline
		$\eps_0 / 2^7$ &3.18E-1 &	1.60E-1 &	8.10E-2 &	4.06E-2 &	2.05E-2 &	1.03E-2 &	5.13E-3\\
		order & -- & 0.99 &	0.99 &	0.99 &	0.99 &	0.99 &	1.00\\\hline\hline
		$\max\limits_{0<\eps\leq 1}e^{\eps, \tau}$ & 3.52E-1 &	1.78E-1 &	8.96E-2 &	4.50E-2 &	2.25E-2 &	1.13E-2 &	5.64E-3\\
		order & -- & 0.98 &	0.99 &	0.99 &	1.00 &	1.00 &	1.00
	\end{tabular*}
	{\rule{\temptablewidth}{1pt}}
	\label{table:LDirac_S1_nrs}
\end{table}

\begin{table}[htp]
	\def\temptablewidth{1\textwidth}
	\caption{Discrete $l^2$ temporal errors $e^{\eps, \tau}(t = 4)$ for the wave function with non-resonant time step size, $S_2$ method. }
	{\rule{\temptablewidth}{1pt}}
	\begin{tabular*}{\temptablewidth}{@{\extracolsep{\fill}}ccccccc}
		$e^{\eps, \tau}(t = 4)$& $\tau_0 = 1 / 2$ & $\tau_0 / 4$ & $\tau_0 / 4^2$ & $\tau_0 / 4^3$ & $\tau_0 / 4^4$ & $\tau_0 / 4^5$\\ \hline
		$\eps_0 = 1 / 2$ & 1.69E-1 & 3.85E-3 &	\textbf{2.36E-4} &	1.47E-5 &	9.20E-7 &	5.75E-8\\
		order & -- & 2.73 & \textbf{2.01} &	2.00 &	2.00 &	2.00\\\hline
		$\eps_0 / 2$ & 9.79E-2 &	1.16E-2 &	4.61E-4 &	\textbf{2.83E-5} &	1.77E-6 &	1.10E-7\\
		order & -- & 1.54 &	2.33 &	\textbf{2.01} &	2.00 &	2.00\\\hline
		$\eps_0 / 2^2$ & 6.76E-2 &	3.93E-3 &	1.32E-3 &	5.76E-5 &	\textbf{3.54E-6} &	2.21E-7\\
		order & -- & 2.05 &	0.78 &	2.26 &	\textbf{2.01} &	2.00\\\hline
		$\eps_0 / 2^3$ & 7.86E-2 &	4.49E-3 &	2.63E-4 &	1.72E-4 &	7.59E-6 &	\textbf{4.67E-7}\\
		order & -- & 2.06 &	2.05 &	0.31 &	2.25 &	\textbf{2.01}\\\hline
		$\eps_0 / 2^4$ & 7.55E-2 &	5.04E-3 &	5.33E-4 &	2.64E-5 &	2.14E-5 &	9.43E-7\\
		order & -- & 1.95 &	1.62 &	2.17 &	0.15 &	2.25\\\hline
		$\eps_0 / 2^5$ & \textbf{7.01E-2} &	1.94E-2 &	2.38E-4 &	6.50E-5 &	3.02E-6 &	2.61E-6\\
		order & \textbf{--} & 0.93 &	3.18 &	0.94 &	2.22 &	0.10\\\hline
		$\eps_0 / 2^7$ & 6.84E-2 &	\textbf{2.67E-3} &	2.77E-4 &	2.31E-4 &	2.76E-6 &	1.04E-6\\
		order & -- & \textbf{2.34} &	1.64 &	0.13 &	3.19 &	0.70\\\hline
		$\eps_0 / 2^9$ & 6.84E-2 &	2.67E-3 &	\textbf{1.65E-4} &	1.03E-5 &	2.08E-6 &	2.10E-6\\
		order & -- & 2.34 &	\textbf{2.01} &	2.00 &	1.15 &	-0.00\\\hline
		$\eps_0 / 2^{11}$ & 6.84E-2 &	2.67E-3 &	1.66E-4 &	\textbf{1.03E-5} &	6.53E-7 &	4.53E-8\\
		order & -- & 2.34 &	2.00 &	\textbf{2.00} &	1.99 &	1.92\\\hline
		$\eps_0 / 2^{13}$ & 6.84E-2 &	2.67E-3 &	1.64E-4 &	1.04E-5 &	\textbf{7.51E-7} &	1.51E-7\\
		order & -- & 2.34 &	2.01 &	1.99 &	\textbf{1.89} &	1.16\\\hline\hline
		$\max\limits_{0<\eps\leq 1}e^{\eps, \tau}(t = 4)$ & 1.69E-1 &	1.94E-2 &	4.11E-3 &	2.31E-4 &	2.14E-5 &	2.61E-6\\
		order & -- & 1.56 &	1.12 &	2.08 &	1.72 &	1.52
	\end{tabular*}
	{\rule{\temptablewidth}{1pt}}
	\label{table:LDirac_S2_nrs}
\end{table}

\begin{figure}[htp]
	\vspace{5pt}
	\caption{Order plot for $S_2$ using non-resonant time steps.}
	\includegraphics[width=0.61\textwidth]{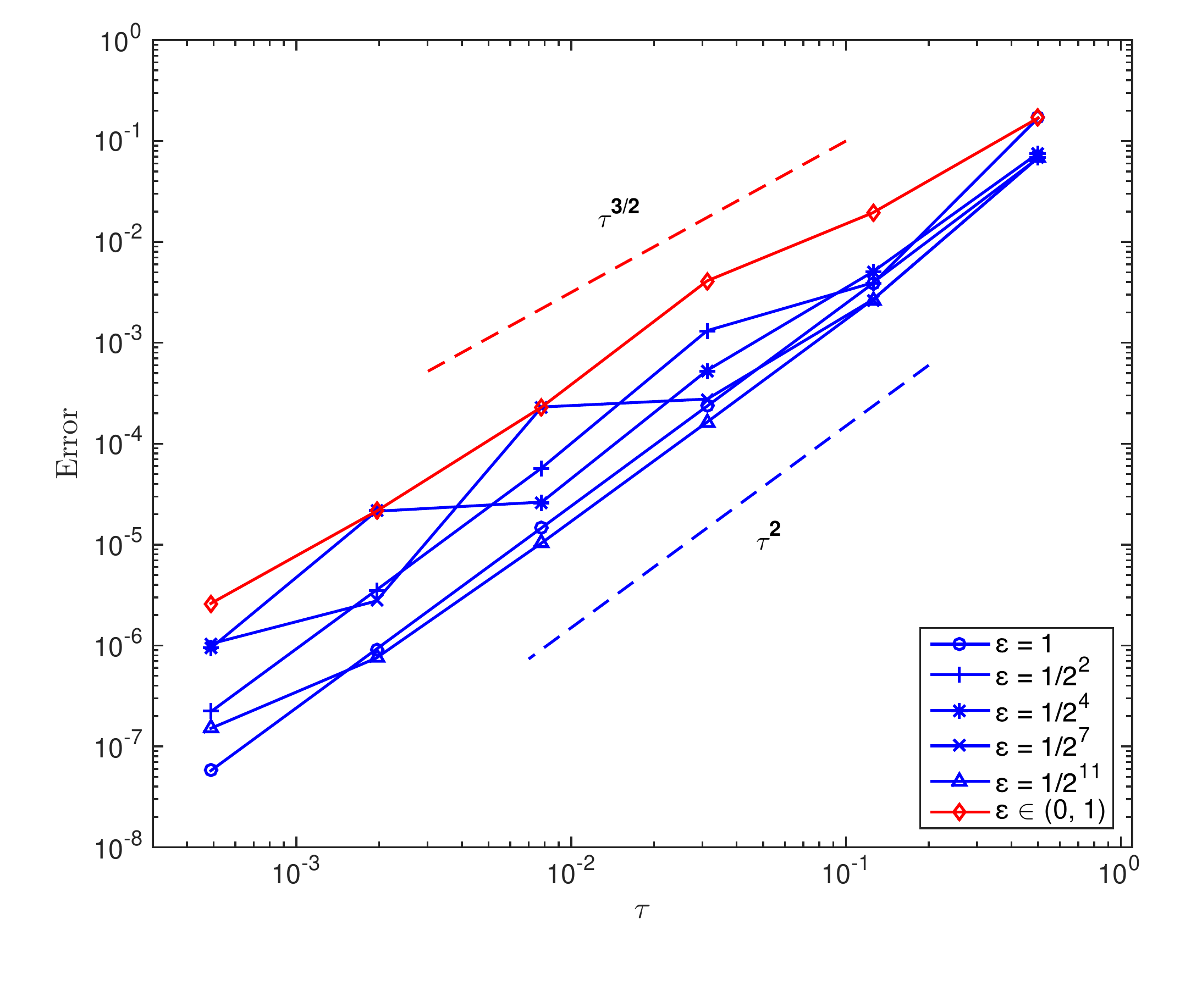}
	\label{fig:S2_nr}
\end{figure}

In Table \ref{table:LDirac_S1_nrs}, we could see that overall, for fixed time step size $\tau$, the error $e^{\eps,\tau}(t = 4)$ does not change with different $\eps$. This verifies the uniform first order convergence in time for $S_1$ with non-resonant time step size, as stated in Theorem \ref{thm:lie2}. In Table \ref{table:LDirac_S2_nrs}, the last two rows show the largest error of each column for fixed $\tau$, which gives $3/2$ order of convergence,  consistent with Theorem \ref{thm:strang2}.
More specifically, in Table \ref{table:LDirac_S2_nrs}, we can observe the second order convergence when $\tau\gtrsim\eps$ (below the lower bolded line) or when $\tau\lesssim\eps^2$ (above the upper bolded line). The lower bolded diagonal  line agrees with the error bound $\|\Phi(t_n,x)-\Phi^n(x)\|_{L^2}\lesssim \tau^2+\tau\eps$, and the upper bolded diagonal line matches the other error bound $\|\Phi(t_n,x)-\Phi^n(x)\|_{L^2}\lesssim \tau^2+\tau^2/\eps$.

Similar to the resonant time step case, Figure \ref{fig:S2_nr} exhibits the order plot for $S_2$ with non-resonant time step sizes. When $\eps$ is relatively large, there is second order uniform convergence for small time step sizes; and when $\eps$ is relatively small, there is second order uniform convergence for large time step sizes. Overall, there is uniform $3/2$ order convergence in time, which corresponds well with Theorem \ref{thm:strang2}.

Through the results of this example, we successfully validate the improved uniform error bounds for the splitting methods in Theorems \ref{thm:lie2} and \ref{thm:strang2}, with non-resonant time step sizes.\\

\noindent\textbf{Example 3} In this example, we deal with a 2D problem. We test the uniform convergence for resonant and non-resonant time step sizes using $S_2$ for \eqref{eq:dirac4_nomag}.

The bounded computational domain is still set as $\Omega = (-10, 10)\times (-10, 10)$.
The numerical `exact' solution is computed by  the $S_2$ method with a very small time step $\tau_e = 2\pi\times 10^{-5}$ for resonant time steps, and $\tau_e = 10^{-5}$ for non-resonant time steps. Spatial mesh size is fixed as $h=1/16$ for all the numerical simulations.

Tables \ref{table:LDirac_S22D_rs} \& \ref{table:LDirac_S22D_nrs} show the numerical errors under resonant and non-resonant time step sizes respectively with different $\varepsilon$.\\

\begin{table}[htp]
	\def\temptablewidth{1\textwidth}
	\caption{Discrete $l^2$ temporal errors $e^{\eps, \tau}(t = 2\pi)$ in 2D for the wave function with resonant time step size, $S_2$ method. }
	{\rule{\temptablewidth}{1pt}}
	
	\begin{tabular*}{\temptablewidth}{@{\extracolsep{\fill}}ccccccc}
		$e^{\eps, \tau}(t = 2\pi)$& $\tau_0 = \pi/16$ & $\tau_0 / 4$ & $\tau_0 / 4^2$ & $\tau_0 / 4^3$
		& $\tau_0 / 4^4$ & $\tau_0 / 4^5$\\ \hline
		$\eps_0 = 1$ &	1.28E-1 &	\textbf{2.56E-3} &	1.57E-4 &	9.78E-6 &	6.08E-7 &	3.41E-8\\
		order & -- & \textbf{2.82} &	2.02 &	2.00 &	2.00 &	2.08\\\hline
		$\eps_0 / 2$ &	4.33E-1 &	6.48E-3 &	\textbf{3.17E-4} &	1.96E-5 &	1.22E-6 &	6.85E-8\\
		order & -- &	3.03 &	\textbf{2.18} &	2.01 &	2.00 &	2.08\\\hline
		$\eps_0 / 2^2$  &	1.01 &	8.71E-2 &	6.99E-4 &	\textbf{3.92E-5} &	2.42E-6 &	1.36E-7\\
		order & -- & 1.76 &	3.48 &	\textbf{2.08} &	2.01 &	2.08\\\hline
		$\eps_0 / 2^3$ & 1.44 &	6.31E-2 &	2.55E-2 &	8.50E-5 &	\textbf{4.88E-6} &	2.73E-7\\
		order & -- &2.26 &	0.65 &	4.11 &	\textbf{2.06} &	2.08\\\hline
		$\eps_0 / 2^4$ &1.46 &	\textbf{5.52E-2} &	1.14E-2 &	9.90E-3 &	1.07E-5 &	\textbf{5.58E-7}\\
		order & -- & \textbf{2.36} &	1.14 &	0.10 &	4.93 &	\textbf{2.13}\\\hline
		$\eps_0 / 2^8$ & 1.46 &	5.22E-2 &	\textbf{3.27E-3} &	5.76E-4 &	5.35E-4 &	5.35E-4\\
		order & -- & 2.40 &	\textbf{2.00} &	1.25 &	0.05 &	0.00\\\hline
		$\eps_0 / 2^{12}$ & 1.46 &	5.22E-2 &	3.22E-3 &	\textbf{2.40E-4} &	1.39E-4 &	1.39E-4\\
		order & -- & 	2.40 &	2.01 &	\textbf{1.87} &	0.39 &	0.00\\\hline
		$\eps_0 / 2^{16}$ &1.46 &	5.22E-2 &	3.22E-3 &	1.99E-4 &	\textbf{1.57E-5} &	5.83E-6\\
		order & -- &2.40 &	2.01 &	2.01 &	\textbf{1.83} &	0.72\\\hline\hline
		$\max\limits_{0<\eps\leq 1}e^{\eps, \tau}$ & 1.46 &	8.71E-2 &	2.55E-2	 & 9.90E-3 &	4.44E-3 &	2.16E-3\\
		order & -- & 2.03 &	0.89 &	0.68 &	0.58 &	0.52
	\end{tabular*}
	{\rule{\temptablewidth}{1pt}}
	\label{table:LDirac_S22D_rs}
\end{table}

\begin{table}[htp]
	\def\temptablewidth{1\textwidth}
	\caption{Discrete $l^2$ temporal errors $e^{\eps, \tau}(t = 4)$ in 2D for the wave function with non-resonant time step size, $S_2$ method. }
	{\rule{\temptablewidth}{1pt}}
	\begin{tabular*}{\temptablewidth}{@{\extracolsep{\fill}}cccccc}
		$e^{\eps, \tau}(t = 4)$& $\tau_0 = 1 / 8$ & $\tau_0 / 8$ & $\tau_0 / 8^2$ & $\tau_0 / 8^3$ & $\tau_0 / 8^4$\\ \hline
		$\eps_0 = 1$ & \textbf{9.41E-3} &	1.43E-4 &	2.23E-6 &	3.47E-8 &	4.92E-10\\
		order & \textbf{--} & 2.01 &	2.00 &	2.00 &	2.05\\\hline
		$\eps_0 / 2^{3/2}$ & 5.54E-2 &	\textbf{3.68E-4} &	5.71E-6 &	8.91E-8 &	1.25E-9\\
		order & -- & \textbf{2.41} &	2.00 &	2.00 &	2.05\\\hline
		$\eps_0 / 2^3$ & 6.56E-1 &	1.23E-3 &	\textbf{1.61E-5} &	2.50E-7 &	3.49E-9\\
		order & -- & 3.02 &	\textbf{2.09} &	2.00 &	2.05\\\hline
		$\eps_0 / 2^{9/2}$ & 3.00E-1 &	3.29E-3 &	5.34E-5 &	\textbf{7.34E-7} &	1.02E-8\\
		order & -- & 2.17 &	1.98 &	\textbf{2.06} &	2.05\\\hline
		$\eps_0 / 2^6$ & \textbf{2.77E-1} &	3.35E-3 &	9.19E-5 &	2.13E-6 &	\textbf{2.64E-8}\\
		order & \textbf{--} & 2.12 &	1.73 &	1.81 &	\textbf{2.11}\\\hline
		$\eps_0 / 2^9$ &2.79E-1 &	\textbf{3.30E-3} &	4.58E-4 &	1.64E-6 &	6.37E-8\\
		order & -- & \textbf{2.13} &	0.95 &	2.71 &	1.56\\\hline
		$\eps_0 / 2^{12}$ & 2.79E-1 &	3.27E-3 &	\textbf{5.08E-5} &	8.57E-7 &	3.09E-7\\
		order & -- & 2.14 &	\textbf{2.00} &	1.96 &	0.49\\\hline
		$\eps_0 / 2^{15}$ & 2.79E-1 &	3.27E-3 &	5.12E-5 &	\textbf{1.24E-6} &	4.45E-7\\
		order & -- & 2.14 &	2.00 &	\textbf{1.79} &	0.49\\\hline\hline
		$\max\limits_{0<\eps\leq 1}e^{\eps, \tau}(t = 4)$ & 6.56E-1 &	1.70E-2 &	4.58E-4 &	1.02E-5 &	4.45E-7\\
		order & -- & 1.76 &	1.74 &	1.83 &	1.51
	\end{tabular*}
	{\rule{\temptablewidth}{1pt}}
	\label{table:LDirac_S22D_nrs}
\end{table}

The conclusions which could be drawn from Table \ref{table:LDirac_S22D_rs} and Table \ref{table:LDirac_S22D_nrs} are similar to those from Table \ref{table:LDirac_S2} and Table \ref{table:LDirac_S2_nrs}. The results validate that the theorems of super-resolution could be extended to 2-dimension, or even higher dimensional cases.

\section{Conclusion}
The super-resolution property of time-splitting methods for the Dirac equation in the nonrelativistic  regime without magnetic potentials were established. We rigorously proved the uniform error bounds, and the improved uniform error bounds with non-resonant time step  for the Lie-Trotter splitting $S_1$ and the Strang splitting $S_2$. For $S_1$,  we have  two independent error bounds $\tau+\eps$ and $\tau + \tau/\eps$, resulting in a uniform 1/2 order convergence. Surprisingly, there will be first order improved uniform convergence if the time step size is non-resonant. For $S_2$, the uniform convergence rate is also 1/2, while the two different error bounds are $\tau^2 + \eps$ and $\tau^2 + \tau^2/\eps^3$ respectively. With non-resonant time step size, the convergence order can be improved to 3/2 for $S_2$, while the two independent error bounds become $\tau^2 + \tau\eps$ and $\tau^2 + \tau^2/\eps$. The numerical results agreed well with the theorems. In this paper,  only  1D case was presented, but indeed the results are still valid in higher dimensions, and the proofs can  be easily generalized. Moreover, higher order time-splitting methods, like the $S_4$, $S_\text{4c}$, $S_\text{4RK}$ methods used in \cite{BY}, also have the super-resolution property for Dirac equation in the nonrelativistic  regime in the absence of external magnetic potentials.

\section*{Appendix}
In this section, we sketch the proofs of Theorems \ref{thm:lie} \&\ref{thm:lie2} for the Lie splitting $S_1$ applied to the four-vector Dirac equation \eqref{eq:dirac4_nomag} in higher dimensions $d=2,3$, as the arguments for the Lie ($S_1$)/Strang splitting ($S_2$)  applied to the four-vector  form \eqref{eq:dirac4_nomag}/two-vector form \eqref{eq:dirac2_nomag} would be similar. In such case, assumptions  (A) and (B)  are directly generalized to the high dimensions ($d=2,3$).

For $d=2,3$, the Lie-Trotter splitting $S_1$ for \eqref{eq:dirac4_nomag}  is
\be\label{eq:Lie-trotter:1}
\Psi^{n+1}(\bx)=e^{-\frac{i\tau}{\eps^2}\mathcal{T}^\eps}e^{-i\int_{t_n}^{t_{n+1}}V(s,\bx)\,ds}\Psi^n(\bx), \quad \bx=(x_1,\ldots,x_d)\in\mathbb{R}^d,
\ee
with $\Psi^0(x) = \Psi_0(x)\in\mathbb{C}^4$, where the free Dirac operator $\mathcal{T}^\eps$ becomes
\begin{equation}
\mathcal{T}^\eps=- \eps\sum_{j = 1}^{d}\alpha_j\partial_j +\beta
\end{equation}
and the decomposition \eqref{eq:dec}  holds with projections $\Pi_\pm^\eps$ given in \eqref{eq:pipm} by replacing $I_2$ with $I_4$. Now, the following expansions for $\Pi_\pm^\eps$ are valid \cite{BMP}
\begin{align}\label{eq:pi+4}
\Pi_+^\eps=\Pi_+^0+\eps\mathcal{R}_1=\Pi_+^0-i\frac{\eps}{2}\sum_{j=1}^d\alpha_j\partial_j+\eps^2\mathcal{R}_2,\quad \Pi_+^0=\text{diag}(1,1,0,0),\\
\label{eq:pi-4}
\Pi_-^\eps=\Pi_-^0-\eps\mathcal{R}_1=\Pi_-^0+i\frac{\eps}{2}\sum_{j=1}^d\alpha_j\partial_j-\eps^2\mathcal{R}_2,\quad \Pi_-^0=\text{diag}(0,0,1,1),
\end{align}
where $\mathcal{R}_1:(H^m({\Bbb R}^d))^4\to (H^{m-1}({\Bbb R}^d))^4$ for $m\geq 1$, $m \in \mathbb{N}^*$, and $\mathcal{R}_2:(H^m({\Bbb R}^d))^4\to (H^{m-2}({\Bbb R}^d))^4$ for $m\geq 2$, $m \in \mathbb{N}^*$ are uniformly bounded operators with respect to $\eps$.

Introduce the error function similar to \eqref{eq:en}
\be\label{eq:e}
{\bf e}^n(\bx) = \Psi(t_n, \bx) - \Psi^n(\bx), \quad 0\leq n\leq \frac{T}{\tau},
\ee
and we will show the  conclusions in Theorems   \ref{thm:lie} \&\ref{thm:lie2}  hold. The proof will be sketched as follows.

(1) {\it Step 1: local error decomposition.} Following the computations in Lemma \ref{lemma:lie}, we have
\be
	{\bf e}^{n+1}(\bx) = e^{-\frac{i\tau}{\eps^2}\mathcal{T}^\eps}e^{-i\int_{t_n}^{t_{n+1}}V(s, \bx)ds}{\bf e}^n(\bx) + \eta_1^n(\bx) + \eta_2^n(\bx), \quad 0\leq n\leq \frac{T}{\tau} - 1,
	\ee
	with $\|\eta_1^n(\bx)\|_{L^2}\lesssim\tau^2$, $\eta_2^n(x) = -ie^{-\frac{i\tau}{\eps^2}\mathcal{T}^\eps}\left(\int_0^\tau f_2^n(s)ds - \tau f_2^n(0)\right)$, where
	\begin{align}\label{eq:f2nlie:A}
	f_2^n(s) = &e^{i2s/\eps^2}e^{is\mathcal{D}^\eps}\Pi_+^\eps\left(V(t_n)\Pi_-^\eps e^{is\mathcal{D}^\eps}\Psi(t_n)\right) \nonumber\\
	&+ e^{-i2s/\eps^2}e^{-is\mathcal{D}^\eps}\Pi_-^\eps\left(V(t_n)\Pi_+^\eps e^{-is\mathcal{D}^\eps}\Psi(t_n)\right).
	\end{align}

(2){\it Step 2: Theorem \ref{thm:lie} for general time steps.} Analogous to the proof of Theorem \ref{thm:lie} in section \ref{sec:general}, the estimates \eqref{eq:pi+eps} and \eqref{eq:pi+eps} hold true for the $d=2,3$ case by noticing the decompositions \eqref{eq:pi+4} and \eqref{eq:pi-4} and $\Pi_\pm^0V(\bx)\Pi_\mp^0=0$. Then the proof of Theorem \ref{thm:lie} for the $d=2,3$ case can be proceeded as the same in section \ref{sec:general}.

(3) {\it Step 3: Theorem \ref{thm:lie2} for non-resonant steps.} Following the proof of Theorem \ref{thm:lie2} for $d=1$ case in section \ref{sec:nonres},  by using the similar estimates as \eqref{eq:pi+4} and \eqref{eq:pi-4} for the $d=2,3$ cases (observed in the above step),  we can derive \eqref{eq:smalle} for the high dimensional cases.  So \eqref{eq:error:lie} is valid. The rest proof for the high dimensional case of Theorem \ref{thm:lie2} ($d=2,3$) can be carried out exactly the same as that in section \ref{sec:nonres}, where only the solution structure \eqref{eq:split_Pauli} of the Dirac equation  is used and such structure is valid in $d=2,3$  \cite{BCJT,BMP}.

As can be seen in the above generalizations to the higher dimensions $d=2,3$, the estimates  \eqref{eq:pi+4} and \eqref{eq:pi-4} play the key roles, which ensures that $\Pi_\pm^\varepsilon V(t,\bx)\Pi_\mp^\eps=O(\varepsilon)$ (valid for  $d=1,2,3$, two-vector form and/or four-vector form). However,  such $O(\varepsilon)$ oder estimates do not hold if electrical potential $V(t,\bx)$ is replaced by the external magnetic potentials and we can only obtain the stated results in this paper for the Dirac equation without magnetic potentials.

\bibliographystyle{amsplain}

\begin{thebibliography}{99}
	\bibitem{AMP}
	{D. A. Abanin, S. V. Morozov, L. A. Ponomarenko, R. V. Gorbachev, A. S. Mayorov, M. I. Katsnelson, K. Watanabe, T. Taniguchi, K. S. Novoselov, L. S. Levitov and A. K. Geim},
	\textit{Giant nonlocality near the Dirac point in graphene},
	Science {\bf 332} (2011), 328-330.
	
	\bibitem{ABB}
	{X. Antoine, W. Bao and C. Besse},
	\textit{Computational methods for the dynamics of the nonlinear Sch\"{o}dinger/Gross-Pitaevskii equations},
	Comput. Phys. Commun. {\bf184} (2013), 2621-2633.
	
	\bibitem{AL}
	{X. Antoine and E. Lorin},
	\textit{Computational performance of simple and efficient sequential and parallel Dirac equation solvers},
	Comp. Phys. Commu. {\bf 220} (2017), 150-172.
	
	\bibitem{ALSFB}
	{X. Antoine, E. Lorin, J. Sater, F. Fillion-Gourdeau and A. D. Bandrauk},
	\textit{Absorbing boundary conditions for relativistic quantum mechanics equations},
	J. Comput. Phys. {\bf 277} (2014), 268-304.
	
	\bibitem{Bad}
	{P. Bader, A. Iserles, K. Kropielnicka and P. Singh},
	\textit{Effective approximation for the linear time-dependent Schr\"{o}dinger equation},
	Found. Comp. Math. {\bf 14} (2014), 689-720.
	
	\bibitem{BCJT}
	{W. Bao, Y. Cai, X. Jia and Q. Tang},
	\textit{A uniformly accurate multiscale time integrator pseudospectral method for the Dirac equation in the
	nonrelativistic limit regime},
	SIAM J. Numer. Anal. {\bf 54} (2016), 1785-1812.
	
	\bibitem{BCJT2}
	{W. Bao, Y. Cai, X. Jia and Q. Tang},
	\textit{Numerical methods and comparison for the Dirac equation in the nonrelativistic limit regime},
	J. Sci. Comput. {\bf 71} (2017), 1094-1134.
	
	\bibitem{BCJY}
	{W. Bao, Y. Cai, X. Jia and J. Yin},
	\textit{Error estimates of numerical methods for the nonlinear Dirac equation in the nonrelativistic limit regime},
	Sci. China Math. {\bf 59} (2016), 1461-1494.
	
	\bibitem{BJM}
	{W. Bao, S. Jin and P. A. Markowich},
	\textit{On time-splitting spectral approximations for the Schr\"{o}dinger equation in the semiclassical regime},
	J. Comput. Phys. {\bf 175} (2002), 487-524.
	
	\bibitem{BJM2}
	{W. Bao, S. Jin and P. A. Markowich},
	\textit{Numerical study of time-splitting spectral discretizations of nonlinear Schr\"{o}dinger equations in the semiclassical regimes},
	SIAM J. Sci. Comput. {\bf 25} (2003), 27-64.
	
	\bibitem{BL}
	{W. Bao and X. Li},
	\textit{An efficient and stable numerical method for the Maxwell-Dirac system},
	J. Comput. Phys. {\bf 199} (2004), 663-687.
	
	\bibitem{BS}
	{W. Bao and F. Sun},
	\textit{Efficient and stable numerical methods for the generalized and vector Zakharov system},
	SIAM J. Sci. Comput. {\bf 26} (2005), 1057-1088.
	
	\bibitem{BSW}
	{W. Bao, F. Sun and G. W. Wei},
	\textit{Numerical methods for the generalized Zakharov system},
	J. Comput. Phys. {\bf 190} (2003), 201-228.
	
	\bibitem{BY}
	{W. Bao and J. Yin},
	\textit{A fourth-order compact time-splitting Fourier pseudospectral method for the Dirac equation},
	Res. Math. Sci. {\bf 6} (2019), article 11.
	
	\bibitem{BMP}
	{P. Bechouche, N. Mauser and F. Poupaud},
	\textit{(Semi)-nonrelativistic limits of the Dirac equation with external time-dependent electromagnetic field},
	Commun. Math. Phys. {\bf 197} (1998), 405-425.
	
	\bibitem{BCLL}
	{O. Boada, A. Celi, J. I. Latorre and M. Lewenstein},
	\textit{Dirac equation for cold atoms in artificial curved spacetimes},
	New J. Phys. {\bf 13} (2011), 035002.
	
	\bibitem{CaiW}
	{Y. Cai and Y. Wang},
	\textit{(Semi)-nonrelativistic limits of the nonlinear Dirac equations}, Journal of Mathematical Study, to appear.
	
	\bibitem{CaiW2}
	{Y. Cai and Y. Wang},
	\textit{Uniformly accurate nested Picard iterative integrators for the Dirac equation in the nonrelativistic limit regime},  SIAM J.  Numer. Anal.  {\bf 57} (2019), 1602-1624.
	
	\bibitem{CHP}
	{E. Carelli, E. Hausenblas and A. Prohl},
	\textit{Time-splitting methods to solve the stochastic incompressible Stokes equation},
	SIAM J. Numer. Anal. {\bf 50} (2012), 2917-2939.
	
	\bibitem{Carles}
	{R. Carles},
	\textit{On Fourier time-splitting methods for nonlinear Schr\"{o}dinger equations in the semiclasscial limit},
	SIAM J. Numer. Anal. {\bf 51} (2013), 3232-3258.
	
	\bibitem{CG}
	{R. Carles and C. Gallo},
	\textit{On Fourier time-splitting methods for nonlinear Schr\"{o}dinger equations in the semi-classical limit II. Analytic regularity},
	Numer. Math. {\bf 136} (2017), 315-342.
	
	\bibitem{Das}
	{A. Das},
	\textit{General solutions of Maxwell-Dirac equations in $1 + 1$ dimensional space-time and spatial confined solution},
	J. Math. Phys. {\bf 34} (1993), 3986-3999.
	
	\bibitem{DK}
	{A. Das and D. Kay},
	\textit{A class of exact plane wave solutions of the Maxwell-Dirac equations},
	J. Math. Phys. {\bf 30} (1989), 2280-2284.
	
	\bibitem{DT}
	{S. Descombes and M. Thalhammer},
	\textit{An exact local error representation of exponential operator splitting methods for evolutionary problems and applications to linear Schr\"{o}dinger equations in the semi-classical regime},
	BIT Numer. Math. {\bf 50} (2009), 729-749.
	
	\bibitem{Dirac}
	{P. A. M. Dirac},
	\textit{The quantum theory of the electron},
	Proc. R. Soc. Lond. A {\bf 117} (1928), 610-624.
	
	\bibitem{ES}
	{M. Esteban and E. S\'{e}r\'{e}},
	\textit{Existence and multiplicity of solutions for linear and nonlinear Dirac problems,
	Partial Differential Equations and Their Applications}, 107-118, 1997.
	
	\bibitem{FJS}
	{D. Fang, S. Jin and C. Sparber},
	\textit{An efficient time-splitting method for the Ehrenfest dynamics},
	Multiscale Model. Simul. {\bf 16} (2018), 900-921.
	
	\bibitem{FW}
	{C. L. Fefferman and M. I. Weistein},
	\textit{Honeycomb lattice potentials and Dirac points},
	J. Am. Math. Soc. {\bf 25} (2012), 1169-1220.
	
	\bibitem{FLB}
	{F. Fillion-Gourdeau, E. Lorin and A. D. Bandrauk},
	\textit{Resonantly enhanced pair production in a simple diatomic model},
	Phys. Rev. Lett. {\bf 110} (2013), 013002.
	
	\bibitem{Gauckler}
	{L. Gauckler},
	\textit{On a splitting method for the Zakharov system},
	Numer. Math. {\bf 139} (2018), 349-379.
	
	\bibitem{GGT}
	{F. Gesztesy, H. Grosse and B. Thaller},
	\textit{A rigorous approach to relativistic corrections of bound state energies for spin-1/2 particles},
	Ann. Inst. Henri Poincar\'{e} Phys. Theor. {\bf 40} (1984), 159-174.
	
	\bibitem{Gross}
	{L. Gross},
	\textit{The Cauchy problem for the coupled Maxwell and Dirac equations},
	Commun. Pure Appl. Math. {\bf 19} (1966), 1-15.
	
	\bibitem{Hairer06}
	{E. Hairer, G. Wanner and C. Lubich},
	\textit{Geometric Numerical Integration},
	Springer-Verlag, 2002.
	
	\bibitem{HJM}
	{Z. Huang, S. Jin, P. A. Markowich, C. Sparber and C. Zheng},
	\textit{A time-splitting spectral scheme for the Maxwell-Dirac system},
	J. Comput. Phys. {\bf 208} (2005), 761-789.
	
	\bibitem{Hunziker}
	{W. Hunziker},
	\textit{On the nonrelativistic limit of the Dirac theory},
	Commun. Math. Phys. {\bf 40} (1975), 215-222.
	
	\bibitem{JL}
	{T. Jahnke and C. Lubich},
	\textit{Error bounds for exponential operator splittings},
	BIT Numer. Math. {\bf 40} (2000), 735-744.
	
	\bibitem{JMZ}
	{S. Jin, P. A. Markowich and C. Zheng},
	\textit{Numerical simulation of a generalized Zakharov system},
	J. Comput. Phys. {\bf 201} (2004), 376-395.
	
	\bibitem{JZ}
	{S. Jin and C. Zheng},
	\textit{A time-splitting spectral method for the generalized Zakharov system in multi-dimensions},
	J. Sci. Comput. {\bf 26} (2006), 127-149.
	
	\bibitem{LLS}
	{S. Li, X. Li and F. Shi},
	\textit{Time-splitting methods with charge conservation for the nonlinear Dirac equation},
	Numer. Meth. Part. D. E. {\bf 33} (2017), 1582-1602.
	
	\bibitem{Lubich}
	{C. Lubich},
	\textit{On splitting methods for Schr\"{o}dinger-Poisson and cubic nonlinear Schr\"{o}dinger equations},
	Math. Comp. {\bf 77} (2008), 2141-2153.
	
	\bibitem{MQ}
	{R. I. McLachlan and G. R. W. Quispel},
	\textit{Splitting methods},
	Acta Numer. (2002) 341-434.
	
	\bibitem{NGPNG}
	{A. H. C. Neto, F. Guinea, N. M. R. Peres, K. S. Novoselov and A. K. Geim},
	\textit{The electronic properties of graphene},
	Rev. Mod. Phys. {\bf 81} (2009), 109-162.
	
	\bibitem{NGMJZ}
	{K. S. Novoselov, A. K. Geim, S. V. Morozov, D. Jiang, Y. Zhang, S. V. Dubonos, I. V. Grigorieva and A. A. Firsov},
	\textit{Electric field effect in atomically thin carbon films},
	Science {\bf 306} (2004), 666-669.
	
	\bibitem{NSG}
	{J. W. Nraun, Q. Su and R. Grobe},
	\textit{Numerical approach to solve the time-dependent Dirac equation},
	Phys. Rev. A {\bf 59} (1999), 604-612.
	
	\bibitem{Ring}
	{P. Ring},
	\textit{Relativistic mean field theory in finite nuclei},
	Prog. Part. Nucl. Phys. {\bf 37} (1996), 193-263.
	
	
	\bibitem{Strang}
	G. Strang,
	\textit{On the construction and comparison of difference schemes},
	SIAM J. Numer. Anal. {\bf 5} (1968), 507-517.
	
	\bibitem{Tha}
	M. Thalhammer,
	\textit{High-order exponential operator splitting
	methods for time-dependent Schr\"{o}dinger equations},
	SIAM J. Numer. Anal. {\bf 46} (2008),  2022-2038.
	
	\bibitem{Trotter}
	H. F. Trotter,
	\textit{On the product of semi-groups of operators},
	Proc. Amer. Math. Soc. {\bf 10} (1959), 545-551.
	
	\bibitem{Verlet}
	L. Verlet,
	\textit{Computer `experiments' on classical fluids, I: Thermodynamical properties of Lennard-Jones molecules},
	Phys. Rev. {\bf 159} (1967), 98-103.
	
	\bibitem{WHJY}
	H. Wu, Z. Huang, S. Jin and D. Yin,
	\textit{Gaussian beam methods for the Dirac equation in the semi-classical regime},
	Commun. Math. Sci. {\bf 10} (2012), 1301-1305.
	
\end{thebibliography}

\end{document}